\documentclass[final, reqno, 11pt]{amsart}
\usepackage{amsmath}
\usepackage{amsfonts}
\usepackage{amssymb}
\usepackage{mathabx}
\usepackage{amsthm}

\newtheorem{theorem}{Theorem}[section]
\newtheorem{lemma}[theorem]{Lemma}
\newtheorem{proposition}[theorem]{Proposition}

\theoremstyle{remark}

\newtheorem{definition}{Definition}[section]
\newcommand{\set}{\mathbb}

\newcommand{\les}{\lesssim}

\newcommand{\mc}{\mathcal}
\newcommand{\be}{\begin{equation}}
\newcommand{\ee}{\end{equation}}
\newcommand{\bee}{\begin{align}}
\newcommand{\eee}{\end{align}}
\newcommand{\ba}{\begin{array}}
\newcommand{\ds}{\displaystyle}

\newcommand{\ea}{\end{array}}
\newcommand{\bpm}{\begin{pmatrix}}
\newcommand{\epm}{\end{pmatrix}}
\newcommand{\lb}{\label}

\DeclareMathOperator{\sgn}{sgn}
\DeclareMathOperator{\supp}{supp}

\DeclareMathOperator{\Imim}{Im}

\DeclareMathOperator{\B}{\mathcal B}
\DeclareMathOperator{\codim}{codim}

\newcommand{\ov}{\overline}
\newcommand{\dd}{{\,}{d}}

\renewcommand{\Im}{\Imim}
\newcommand{\R}{\mathbb R}
\newcommand{\C}{\mathbb C}

\newcommand{\V}{\mathcal V}
\newcommand{\W}{\mathcal W}

\title[Threshold Eigenstates and Resonances]{Dispersive estimates in $\R^3$ with Threshold Eigenstates and Resonances}
\author{Marius Beceanu}
\address{Institute for Advanced Study, Einstein Drive, Princeton, NJ, 08540, USA}
\email{mbeceanu@ias.edu}
\thanks{The author was partially supported by a Rutgers Research Council grant.}
\begin{document}
\maketitle
\numberwithin{equation}{section}
\begin{abstract} We prove dispersive estimates in $\R^3$ for the Schr\"{o}dinger evolution generated by the Hamiltonian $H = -\Delta+V$, under optimal decay conditions on $V$, in the presence of zero energy eigenstates and resonances.
\end{abstract}

\tableofcontents
\section{Introduction}
\subsection{Classification of exceptional Hamiltonians} Consider a Hamiltonian of the form $H=-\Delta+V$, where $V$ is a real-valued scalar potential on~$\R^3$.

We assume $V \in L^{3/2, 1} \subset L^{3/2}$, which is the predual of weak-$L^3$ and a Lorentz space, $L^{3/2, 1} \subset L^{3/2-\epsilon} \cap L^{3/2+\epsilon}$; for its definition and properties see \cite{bergh}. By \cite{simon} this is sufficient to guarantee the self-adjointness of $H=-\Delta+V$.

Let $R_0(\lambda): = (-\Delta - \lambda)^{-1}$ be the free resolvent corresponding to the free evolution $e^{-it\Delta}$ and let $R_V(\lambda) := (-\Delta+V-\lambda)^{-1}$ be the perturbed resolvent corresponding to the perturbed evolution $e^{itH}$. Explicitly, in three dimensions and for $\Im \lambda \geq 0$, 
\be\lb{eq_3.55}
R_0((\lambda+i0)^2)(x, y) = \frac 1 {4\pi} \frac {e^{i \lambda |x-y|}}{|x-y|}.
\ee

It will be shown below that under reasonable assumptions $H$ has only finitely many negative eigenvalues. Then the Schr\"{o}dinger evolution restricted to the continuous spectrum $[0, \infty)$ 
has the representation formula
$$\begin{aligned}
e^{itH} P_c = \lim_{\epsilon \to 0} \frac 1 {2\pi i} \int_0^{\infty} e^{it\eta} \big(R_V(\eta+i\epsilon) - R_V(\eta-i\epsilon)\big) \dd \eta.
\end{aligned}$$

By the work of Ionescu--Jerison \cite{ionjer} and Goldberg--Schlag \cite{golsch}, it is known that, when $V \in L^{3/2}$, $R_V(\lambda \pm i0)$ is uniformly bounded in $\B(L^{6/5}, L^6)$ on any interval $\lambda \in [\epsilon_0, \infty)$, $\epsilon_0>0$, and has no singularities in $[0, \infty)$ except potentially at $\lambda=0$.

Observe that $R_V = (I+R_0V)^{-1}R_0$, so $R_V$ has a singularity at zero precisely when $I + R_0(0) V$, which is compact, is not invertible.

We denote the null space of $I + R_0(0) V$ by $\mc M$:
$$
\mc M := \{\phi \in L^{\infty} \mid \phi + R_0(0) V \phi = 0\}.
$$
If $\mc M \ne \varnothing$ we say that $H$ is of \emph{exceptional type}, while if $\mc M = \varnothing$ we say that $H$ is of \emph{generic type}.

The sesquilinear form $-\langle u, V v \rangle$ is an inner product on $\mc M$, see Lemma \ref{lemma_11'}. This pairing is well-defined when $V \in L^{3/2, 1}$ because $u, v \in L^{3, \infty} \cap L^\infty$ by Lemma \ref{lemma_11}.


Let $\mc E := \mc M \cap L^2$ and $P_0$ be the orthogonal $L^2$ projection onto $\mc E$. In Lemma \ref{lemma_12} we provide a characterization of $\mc E$ and show that $\codim_{\mc M} \mc E \leq 1$.

The set $\mc E_1 := \mc E \cap L^1$ also plays a special part in the proof. In Lemma \ref{decay_lemma} we give a characterization of $\mc E_1$ and prove that $\codim_{\mc E} \mc E_1 \leq 12$.

A function $\phi \in \mc M \setminus \mc E$ is called a zero energy resonance of $H$. Following \cite{jeka} and \cite{yajima_disp}, we classify exceptional Hamiltonians $H$ as follows:
\paragraph{1.} $H$ is of exceptional type of the first kind if it has a zero energy resonance, but no zero energy eigenfunctions: $\{0\} = \mc E \subsetneq \mc M$.
\paragraph{2.} $H$ is of exceptional type of the second kind if it has zero energy eigenfunctions, but no zero energy resonance: $\{0\} \subsetneq \mc E = \mc M$.
\paragraph{3.} $H$ is of exceptional type of the third kind if it has both resonances and eigenfunctions at zero energy: $\{0\} \subsetneq \mc E \subsetneq \mc M$.


\subsection{Main result}
When $H$ is of exceptional type of the first kind, we let the \emph{canonical resonance} be $\phi \in \mc M$ such that $\langle V, \phi \rangle > 0$ and $-\langle \phi, V \phi \rangle = 1$ (one can make these choices by Lemma \ref{lemma_12} and Lemma \ref{lemma_11'}, respectively).


Using the canonical resonance $\phi(x)$, we define a constant $a$ and a function $\zeta_t(x)$ by 
$$
a = \frac{4\pi i} {|\langle V , \phi \rangle|^2},\ \zeta_t(x) = e^{\frac{i|x|^2}{4t}} \phi(x).
$$
We also define a function $\mu_t(x)$ by
$$\ds\mu_t(x) := \frac i {|x|} \int_0^1 (e^{\frac {i|x|^2}{4t}} - e^{\frac {i|\theta x|^2}{4t}}) \dd \theta.
$$
Let the operators $R(t)$ and $S(t)$ be given by
\be\lb{rs}\begin{aligned}
&R(t) := \frac{ae^{-i\frac{3\pi}4}}{\sqrt{\pi t}} \zeta_t(x) \otimes \zeta_t(y),\\
&S(t) := \frac{e^{-i\frac {3\pi} 4}}{\sqrt{\pi t}}\Big(-i P_0 V \frac{|x-y|^2}{24\pi} V P_0 + \mu_t(x) \frac{|x-y|}{8\pi} V P_0 + P_0 V \frac{|x-y|}{8\pi} \mu_t(y)\Big).
\end{aligned}\ee
Note that
$$
\|R(t) u\|_{L^{3, \infty}} + \|S(t) u\|_{L^{3, \infty}} \les t^{-1/2} \|u\|_{L^{3/2, 1}}.
$$

\begin{proposition}[Main result]\lb{main_result} Assume that $\langle x \rangle^2 V \in L^{3/2, 1}$ and that $H=-\Delta+V$ is exceptional of the first kind. Then for $1 \leq p < 3/2$ and any $u \in L^2 \cap L^p$
$$\begin{aligned}
&e^{-itH} P_c u = Z(t) u + R(t)u,\ \|Z(t)u\|_{L^{p'}} \les t^{-\frac 3 2(\frac 1 p - \frac 1 {p'})} \|f\|_{L^p},
\end{aligned}$$
where $p'$ is the dual exponent $\frac 1 p + \frac 1 {p'} = 1$. Furthermore, assuming only that $V \in L^{3/2, 1}$, for $3/2<p\leq 2$
$$
\|e^{-itH} P_c u\|_{L^{p'}} \les t^{-\frac 3 2(\frac 1 p - \frac 1 {p'})} \|u\|_{L^p},\ \|e^{-itH} P_c u\|_{L^{3, \infty}} \les t^{-1/2} \|u\|_{L^{3/2, 1}}.
$$
Assume that $\langle x \rangle^4 V \in L^{3/2, 1}$ and that $H=-\Delta+V$ is exceptional of the second or third kind. Then for $1 \leq p < 3/2$ and any $u \in L^2 \cap L^p$
\be\lb{dispersie}
e^{-itH} P_c u = Z(t) u + R(t) u + S(t) u,\ \|Z(t) u\|_{L^{p'}} \les t^{-\frac 3 2 (\frac 1 p -\frac 1 {p'})} \|u\|_{L^p},
\ee
where $R(t)$ is missing if $H$ is an exceptional Hamiltonian of the second kind.

In the case when all the zero energy eigenfunctions of $H$ are in $L^1$, one can omit $S(t)$ from (\ref{dispersie}).

Assume that $\langle x \rangle^2 V \in L^{3/2, 1}$ and that $H=-\Delta+V$ is exceptional of the second or third kind. Then for $3/2<p\leq 2$
$$
\|e^{-itH}P_c u\|_{L^{p'}} \les t^{-\frac 3 2(\frac 1 p - \frac 1 {p'})} \|u\|_{L^p},\ \|e^{-itH}P_c u\|_{L^{3, \infty}} \les t^{-1/2} \|u\|_{L^{3/2, 1}}.
$$
\end{proposition}

Note that, in terms of powers of $x$, the decay conditions on the potential correspond to $|V| \les \langle x \rangle^{-2-}$, $|V| \les \langle x \rangle^{-4-}$, and $|V| \les \langle x \rangle^{-6-}$.

Also note that these decay estimates also imply a certain range of Strichartz estimates.

The rest of the paper is dedicated to proving this main result, which is a combination of Propositions \ref{prop28}, \ref{prop29}, \ref{prop_3}, and \ref{prop219}. For brevity, we omit the proof in case $H$ is an exceptional Hamiltonian of the second kind, which is similar to the case when $H$ is exceptional of the third kind.

\subsection{History of the problem} We study solutions to the linear Schr\"{o}dinger equation in $\R^3$ with potential
$$
i \partial_t u + \Delta u - V u = 0,\ u(0) \text{ given}.
$$

By the RAGE theorem, 
every solution is the sum of a bound and a scattering component. The quantitative study of scattering states began with Rauch \cite{rau}, who proved that if $H = -\Delta+gV$, $g \in \C$, with exponentially decaying $V$, then $e^{itH} P_c$ has a local decay rate of $t^{-3/2}$, with at most a discrete set of exceptional $g$ for which the decay rate is $t^{-1/2}$. Here $P_c$ is the projection on the space of scattering solutions.

Threshold estimates in the presence of eigenvalues and resonances go back to the work of Jensen--Kato \cite{jeka}, who obtained an asymptotic expansion of the resolvent $R(\zeta) = (H-\zeta)^{-1}$ into
$$
R(\zeta)= -\zeta^{-1} B_{-2} - i\zeta^{-1/2}B_{-1} + B_0 + i\zeta^{1/2}B_1 + \ldots
$$
and similar ones for the spectral density and the $S$-matrix. The condition imposed on the potential was polynomial decay at infinity of the form $(1 + |x|^\beta)V(x)\in L^{3/2}(\R^3)$, $\beta>2$.

The possible singularities in this expansion are due to the presence of resonances or eigenstates at zero. $B_{-2}$ is the $L^2$ orthogonal projection on the zero eigenspace, while $B_{-1}$ is given by
$$
B_{-1} = P_0 V \frac{|x-y|^2}{24 \pi} V P_0 - \phi \otimes \phi,
$$
where $\phi$ is the canonical zero resonance, see 
above.

Jensen--Kato also obtained an asymptotic expansion for the evolution $e^{itH} P_c$, in two cases: if zero is a regular point, then
$$
e^{itH} P_c = -(4\pi i)^{-1/2} t^{-3/2} B_0 + o(t^{-3/2})
$$
and if there is only a resonance $\phi$ at zero then
$$
e^{itH} P_c = (\pi i)^{-1/2} t^{-1/2} \phi \otimes \phi + o(t^{-1/2}).
$$

Murata \cite{mur} extended these results by obtaining an asymptotic expansion to any order, for a more general evolution, with or without singular points, and then proving that each term in the expansion is degenerate. Murata's expansion and proof are valid in weighted $L^2$ spaces.

Erdogan--Schlag \cite{ersc1} obtained an asymptotic expansion of the evolution $e^{itH} P_c$ in the pointwise $L^1$-to-$L^\infty$ setting using the Jensen--Nenciu \cite{jene} lemma. The condition assumed for the potential was that $|V(x)| \les \langle x \rangle^{-12-\epsilon}$. The same method works in the case of nonselfadjoint Hamiltonians, see \cite{ersc2}, of the form
$$
\mc H = \bpm -\Delta +\mu + V_1 & V_2 \\ -V_2 & \Delta-\mu-V_1 \epm,
$$
assuming that $|V_1(x)| + |V_2(x)| \les \langle x \rangle^{-10-\epsilon}$.

At the same time, Yajima \cite{yajima_disp} proved a similar expansion for generic Hamiltonians $H=-\Delta+V$ if $|V(x)| \leq \langle x \rangle^{-5/2-\epsilon}$, for singular Hamiltonians of the first kind when $|V(x)| \leq \langle x \rangle^{-9/2-\epsilon}$, and of the second and third kind when $|V(x)| \leq \langle x \rangle^{-11/2-\epsilon}$. His main result stated the following:
\begin{theorem}[Theorem 1.3, \cite{yajima_disp}]\begin{list}{\labelitemi}{\leftmargin=1em}
\item[(1)] Let V satisfy $|V(x)| \leq C\langle x \rangle^{-\beta}$ for some $\beta > 5/2$. Suppose that $H$ is of generic type. Then, for any $1 \leq q \leq 2 \leq p \leq \infty$ such that $1/p + 1/q = 1$,
\be\lb{1.16}
\|e^{-itH}P_c u\|_p \leq C_p t^{-3(\frac 1 2 - \frac 1 p)} \|u\|_q, u \in L^2 \cap L^q.
\ee
\item[(2)] Let $V$ satisfy $|V(x )| \leq C\langle x \rangle^{-\beta}$ for some $\beta > 11/2$. Suppose that $H$ is of exceptional type. Then the following statements are satisfied:
\begin{list}{\labelitemi}{\leftmargin=1em}
\item[a)] Estimate (\ref{1.16}) holds when $p$ and $q$ are restricted to $3/2 < q 
\leq 2 \leq p < 3$ and $1/p + 1/q = 1$. 
\item[b)] Estimate (\ref{1.16}) holds when $p = 3$ and $q = 3/2$ provided that $L^3$ and $L^{3/2}$ are respectively replaced by Lorentz spaces $L^{3,\infty}$ and $L^{3/2, 1}$. 
\item[c)] When $3 < p \leq \infty$ and $1 \leq q < 3/2$ are such that $1/p + 1/q = 1$, there exists a constant $C_{pq}$ such that for any $u \in L^2 \cap L^q$, 
$$
\|(e^{-i t H} P_c - R(t) - S (t)) u\|_p \les C_{pq} t^{-3(\frac 1 2 - \frac 1 p)} \|u\|_q.
$$
If $H$ is of exceptional type of the first kind, statement (2) holds under a weaker decay condition $|V(x)| \leq C \langle x \rangle^{-\beta}$ with $\beta > 9/2$.
\end{list}
\end{list}
\end{theorem}

However, note that due to a mistake in the proof the requirement $\beta>11/2$ should be replaced by $\beta>8$.

When the zero energy eigenfunctions $\phi_k$ of $H$ have enough decay, both $R(t)$ and $S(t)$ can be taken to be zero. Indeed, in \cite{gol2} Goldberg showed that if $V \in L^{3/2-\epsilon} \cap L^{3/2+\epsilon}$ and the zero energy eigenfunctions are in $L^1$ then $\|e^{-itH} P_c u\|_{L^\infty} \les t^{-3/2} \|u\|_{L^1}$. We retrieve a similar result in our context.

Some of our results for exceptional potentials of the first kind hold under the same decay assumption as those for generic potentials: $V \in L^{3/2, 1}$. A similar fact was also recently noticed by Egorova--Kopylova--Marchenko--Teschl \cite{ekmt} in dimension one.

Several results \cite{jss} \cite{gosc} \cite{gol} \cite{becgol} address the issue of pointwise decay in the case of generic Hamiltonians --- for $L^{3/2-\epsilon} \cap L^{3/2+\epsilon}$ potentials in \cite{gol} and Kato-class potentials in \cite{becgol}.

Results obtained in other dimensions include \cite{ccv}, \cite{ekmt}, \cite{egg}, \cite{ergr1}, \cite{ergr2}, \cite{ergr3}, \cite{ergr4}, \cite{gol3}, \cite{gogr1}, \cite{gogr2}, \cite{green}, and~\cite{schlag}.

The current result, Theorem \ref{main_result}, represents an improvement on \cite{yajima_disp}, by half a power of potential decay for exceptional Hamiltonians of the first kind. We expect the rate of potential decay from Theorem \ref{main_result} to be optimal for this sort of result.

The same considerations apply in the case of exceptional Hamiltonians of the second and third kind, also leading to similar improved results. These will constitute the subject of a separate paper.

Below we mostly follow the scheme of Yajima's proof in \cite{yajima_disp}, making the changes from H\"{o}lder spaces to Wiener spaces needed to improve the result. The proof method that we use here is the same as in \cite{bec} and \cite{becgol}.

\section{Proof of the statements}
\subsection{Notations}

We denote by $L^p$ the usual Lebesgue spaces and by $L^{p, q}$, $1 \leq p, q \leq \infty$, the Lorentz spaces. Note here that $L^{p, p}=L^p$, $L^{p, \infty}$ is weak-$L^p$, and $L^{p, q_1} \subset L^{p, q_2}$ for $q_1 \leq q_2$. For the definition and further properties see \cite{bergh}.

Let Sobolev spaces be $W^{s, p}$, $s \in \set R$, $1 \leq p \leq \infty$ and denote weighted Lebesgue spaces by $f(x) L^p = \{f(x) g(x) \mid g \in L^p\}$.

Fix the Fourier transform to
$$
\widehat f(\xi) = \int_{\set R^d} e^{-ix \xi} f(x) \dd x,\ \widecheck f(x) = (2\pi)^{-d} \int_{\set R^d} e^{i\xi x} f(\xi) \dd \xi.
$$
Let $R_0(\lambda):=(-\Delta-\lambda)^{-1}$ and for $\lambda \in \R$ $R_{0a}(\lambda):=\frac 1 {i} (R_0(\lambda+i0)-R_0(\lambda-i0))$. Concerning the Fourier transform, resolvents, and the free evolution, note that with our definitions
$$\begin{aligned}
e^{itH_0} &= (R_{0a}(\lambda))^{\vee}(t);\ R_{0a}(\lambda) = (e^{itH_0})^{\wedge},\ \lambda \in \R;\\
i R_0(\lambda) &= (\chi_{[0, \infty)}(t) e^{itH_0})^{\wedge}(\lambda),\ \Im \lambda < 0.
\end{aligned}$$
Likewise let $R_V(\lambda) := (-\Delta+V-\lambda)^{-1}$.

Also, let
\begin{list}{\labelitemi}{\leftmargin=1em}
\item[$\ast$] $\chi_A$ be the characteristic function of the set $A$;
\item[$\ast$] $\mc M$ be the space of finite-mass Borel measures on $\set R$;
\item[$\ast$] $\delta_x$ denote Dirac's measure at $x$;
\item[$\ast$] $\langle x \rangle = (1+|x|^2)^{\frac 1 2}$;
\item[$\ast$] $\B(X, Y)$ be the Banach space of bounded operators from $X$ to $Y$ and $\B(X)$ be the Banach space of bounded operators from $X$ to itself;
\item[$\ast$] $C$ be any constant (not always the same throughout the paper);
\item[$\ast$] $a \les b$ mean $|a| \leq C |b|$;
\item[$\ast$] $\mc S$ be the Schwartz space;
\item[$\ast$] $u \otimes v$ mean the rank-one operator $\langle \cdot, v \rangle u$;
\item[$\ast$] $K(x, y)$ denote the operator having $K(x, y)$ as integral kernel.
\end{list}

For a potential $V$, let $V_1 = |V|^{1/2}$ and $V_2 = |V|^{1/2} \sgn V$.

\subsection{Auxiliary results}
Recall that $\mc M$ is the kernel of $I+R_0(0) V$ in $L^\infty$.
\begin{lemma}\lb{lemma_11}
Let $V \in L^{3/2, 1}$; then $\mc M \subset L^{3, \infty}$. Conversely, any $\phi \in L^{3, \infty}$ that satisfies the equation $\phi + R_0(0) V \phi = 0$ must be in $L^\infty$, hence in $\mc M$.
\end{lemma}
\begin{proof}[Proof of Lemma \ref{lemma_11}]
Let $V = V^1 + V^2$, where $V^1$ is smooth of compact support and $\|V^2\|_{L^{3/2, 1}} <<1$. Then, if $\phi$ solves the equation,
$$\begin{aligned}
\phi &= -(I + R_0(0)V^2)^{-1} R_0(0) V^1 \phi\\
&= -\Big(\sum_{k=0}^\infty (-1)^k (R_0(0)V^2)^k\Big) R_0(0) V^1 \phi.
\end{aligned}$$
where the inverse is the sum of a Neumann series, thus bounded on $L^{3, \infty}$ and on $L^\infty$.

If $\phi \in L^\infty$, then $V^1 \phi \in L^1$, hence $R_0(0) V^1 \phi \in L^{3, \infty}$, so $\phi \in L^{3, \infty}$.

If $\phi \in L^{3,\infty}$, then $V^1 \phi \in L^{3/2, 1}$, hence $R_0(0) V^1 \phi \in L^{\infty}$, so $\phi \in L^{\infty}$.
\end{proof}

\begin{lemma}\lb{lemma_11'} The quadratic form $-\langle u, V v\rangle$ is an inner product on $\mc M$.
\end{lemma}
\begin{proof} Suppose $u$, $v \in \mc M$. By the definition of $\mc M$, observe that $-\langle u, V v \rangle = \langle u, -\Delta v \rangle$, where $u \in L^{3, \infty} \cap L^\infty$ by Lemma \ref{lemma_11} and $-\Delta v = V v \in L^1 \cap L^{3/2, 1}$. Thus the pairing is well-defined.

Furthermore, $\nabla u = \nabla R_0(0) V u \in L^{3/2, \infty} \cap L^{3, \infty} \subset L^2$ and same for $\nabla v$, so their pairing is also well-defined and we can write $(u, -\Delta v) = (\nabla u, \nabla v)$.

This expression is positively defined because, setting $u=v$, $\langle\nabla u, \nabla u\rangle=0$ implies that $u$ is constant, hence, in view of the fact that $u \in L^{3, \infty}$ by Lemma \ref{lemma_11}, $u=0$.
\end{proof}

Recall that $\mc E = \mc M \cap L^2$.
\begin{lemma}\lb{lemma_12}
Assume that $V \in L^{3/2, 1}$. Then, for any $\phi \in \mc M$, $\phi(x) \in~\langle x \rangle^{-1} L^\infty$.

Assume that $V \in L^1 \cap L^{3/2, 1}$. Then, for any $\phi \in \mc M$, $\ds\phi(x) - \frac {\langle \phi, V \rangle} {4\pi|x|} \in |x|^{-1} L^{3, \infty} \cap |x|^{-1} L^\infty \subset L^2$. Thus $\phi \in \mc M$ is in $\mc E$ if and only if $\langle \phi, V \rangle = 0$; thus $\codim_{\mc M} \mc E \leq 1$. Also, $\mc E \subset \langle x \rangle^{-2} L^\infty$.
\end{lemma}
\begin{proof}[Proof of Lemma \ref{lemma_12}] First, assume that $V \in L^{3/2, 1}$. Rewrite the eigenfunction equation
$$
\phi(x) = - \frac 1 {4\pi} \int_{\R^3} \frac 1 {|x-y|} V(y) \phi(y) \dd y
$$
as
$$\begin{aligned}
|x| \phi(x) + \frac 1 {4\pi} \int_{|y| \geq R} \frac {|x|-|x-y|}{|x-y||y|} V(y) |y| \phi(y) \dd y = -\frac 1 {4\pi} \int_{\R^3} V(y) \phi(y) \dd y - \\
\frac 1 {4\pi} \int_{|y| \leq R} \frac {|x|-|x-y|}{|x-y|} V(y) \phi(y) \dd y.
\end{aligned}$$
Note that $||x|-|x-y||\leq |y|$ and $\lim_{R \to \infty} \|\chi_{|x| \geq R}(x) V(x)\|_{L^{3/2, 1}} = 0$. Then, for sufficiently large $R$ we can invert
$$
(T_0 \phi)(x) = \phi(x)+\frac 1 {4\pi} \int_{|y| \geq R} \frac {|x|-|x-y|}{|x-y||y|} V(y) \phi(y) \dd y
$$
as an operator in $\B(L^\infty)$. Since $\phi(y) \in L^{3, \infty} \cap L^\infty$, the right-hand side is in $L^\infty$, so we obtain that $|x| \phi(x) \in L^\infty$.

Next, assume that $V \in L^1 \cap L^{3/2, 1}$. Start from
$$\begin{aligned}
\phi(x) - \frac {\langle \phi, V \rangle} {4\pi|x|} &= -\frac 1 {4\pi} \int_{\R^3} \bigg(\frac 1 {|x-y|} - \frac 1 {|x|}\bigg) V(y) \phi(y) \dd y\\
&= -\frac 1 {4\pi |x|} \int_{\R^3} \frac {|x|-|x-y|} {|x-y|} V(y) \phi(y) \dd y
\end{aligned}$$
which is bounded in absolute value by $\ds \frac 1 {4\pi|x|} \int_{\R^3} \frac{|y| |V(y)| |\phi(y)|}{|x-y|} \dd y$. Since $\phi \in \langle x \rangle^{-1} L^\infty$ and $V \in L^1 \cap L^{3/2, 1}$, this expression is in $|x|^{-1} L^\infty \cap |x|^{-1} L^{3, \infty} \subset \langle x \rangle^{-1} L^{3, \infty} \subset L^2$.

Since whenever $\langle \phi, V \rangle \ne 0$ $\frac {\langle \phi, V \rangle} {4\pi|x|} \not \in L^2$, it follows that it is necessary and sufficient for $\phi$ to be in $L^2$ that $\langle \phi, V \rangle = 0$.

The space $\mc E$ is then the kernel of the rank-one map $\phi \mapsto \langle \phi, V\rangle$ from $\mc M$ to $\C$, so it has codimension at most $1$.

Finally, we already know that $\mc E \subset \mc M \subset \langle x \rangle^{-1} L^\infty$. The eigenfunction equation for a function $\phi \in \mc E$, for which $\langle \phi, V \rangle = 0$, can be written as
$$
\phi(x) = -\frac 1 {4\pi} \int_{\R^3} \frac {|x|-|x-y|}{|x-y||x|} V(y) \phi(y) \dd y.
$$
We further rewrite it as
$$\begin{aligned}
&|x|^2\phi(x) + \frac 1 {4\pi} \int_{|y| \geq R} \frac {(|x|-|x-y|)^2}{|x-y||y|^2} V(y) |y|^2 \phi(y) \dd y = \\
&= -\frac 1 {4\pi} \int_{\R^3} (|x|-|x-y|) V(y) \phi(y) \dd y - \frac 1 {4\pi} \int_{|y| \leq R} \frac {(|x|-|x-y|)^2}{|x-y|} V(y) \phi(y) \dd y.
\end{aligned}$$
The right-hand side is in $L^\infty$ and for sufficiently large $R$ the left-hand side is invertible, as above. This shows that $|x|^2 \phi(x) \in L^\infty$.
\end{proof}

We can continue the asymptotic expansion of eigenfunctions to any order, but first we need the following lemma.
\begin{lemma}\lb{inegal} For $x$, $y \in \R^3$
\be\lb{inequality}
\Big|\frac 1 {|x-y|} - \Big(\frac 1 {|x|} + \frac {xy}{|x|^3}\Big)\Big| \les \frac {|y|^2}{|x|^2|x-y|}
\ee
and
\be\lb{inequality2}
\Big|\frac 1 {|x-y|} - \Big(\frac 1 {|x|} + \frac {xy}{|x|^3} + \frac {|y|^2}{2|x|^3} - \frac {3(xy)^2}{2|x|^5}\Big)\Big| \les \frac {|y|^3}{|x|^3|x-y|}.
\ee
\end{lemma}
More generally, it seems to be the case (one can prove by induction) that
$$
\Big|\frac 1 {|x+y|} - \sum_{k=0}^N d^k \frac 1 {|x-\cdot|} (y, \ldots, y)\Big| \les \frac {|y|^{N+1}}{|x|^{N+1}|x-y|}.
$$
\begin{proof}[Proof of Lemma \ref{inegal}]
Indeed, we start from
\be\lb{eq_1}
(|x|^2+2xy+|y|^2)^{1/2}-(|x|^2)^{1/2} = \frac {2xy}{|x+y|+|x|} + \frac {|y|^2}{|x+y|+|x|}.
\ee
Then
$$
\Big|\frac {2xy}{|x+y|+|x|} - \frac {xy}{|x|}\Big| = \Big|\frac {xy(|x|-|x+y|)}{(|x+y|+|x|)|x|}\Big| \les \frac{|y|^2}{|x|}.
$$
Therefore
\be\lb{ineg_1}
\Big||x+y|-|x|-\frac {xy}{|x|}\Big| \les \frac {|y|^2}{|x|}.
\ee
Consequently
$$
\big||x|^2(|x|-|x-y|) - xy |x-y|\big| \leq |x|^2\Big||x-y|-|x|+ \frac {xy}{|x|}\Big| + |xy(|x|-|x-y|)| \les |y|^2|x|.
$$
Dividing by $|x|^3|x-y|$ we obtain (\ref{inequality}).

We next perform a more detailed analysis of the same inequality. In (\ref{eq_1}), by (\ref{ineg_1})
$$\begin{aligned}
\Big|\frac{xy(|x|-|x+y|)}{(|x+y|+|x|)|x|} + \frac {(xy)^2}{2|x|^3}\Big| &\les \Big|\frac{xy(|x|-|x+y|)}{(|x+y|+|x|)|x|} - \frac {xy(|x|-|x+y|)}{2|x|^2}\Big| +\\
&+ \Big|\frac {xy(\frac{xy}{|x|}+|x|-|x+y|)}{2|x|^2}\Big| \les \frac {|y|^3}{|x|^2}.
\end{aligned}$$
Furthermore, also in (\ref{eq_1}),
$$
\frac {|y|^2}{|x+y|+|x|} - \frac {|y|^2}{2|x|} \les \frac {|y|^3}{|x|^2}.
$$
Therefore
\be\lb{ineg_2}
\Big||x+y|-|x|-\frac{xy}{|x|}-\frac{|y|^2}{2|x|}+\frac{(xy)^2}{2|x|^3}\Big| \les \frac {|y|^3}{|x|^2}.
\ee
By (\ref{ineg_1}) and (\ref{ineg_2}) we then obtain (\ref{inequality2}).
\end{proof}

We can now establish the asymptotic expansion of eigenfunctions.
\begin{lemma}\lb{decay_lemma} Assume that $V \in L^1 \cap L^{3/2, 1}$. Let $\phi \in \mc E$ be a zero energy eigenfunction of $H$. Then 
$$
\phi(x) - \sum_{k=1}^3 \langle V \phi, y_k \rangle \frac {x_k}{|x|^3} \in |x|^{-2} (L^{3, \infty} \cap L^\infty).
$$
Further assume that $V \in \langle x \rangle^{-1} L^1 \cap L^{3/2, 1}$. Then
$$
\phi(x) - \sum_{k=1}^3 \langle V \phi, y_k \rangle \frac {x_k}{|x|^3} - \sum_{k, \ell=1}^3 \langle \phi V, y_k y_\ell \rangle \Big(\frac {\delta_{k\ell}}{2|x|^3} - \frac {3 x_k x_\ell}{2|x|^5} \Big) \in |x|^{-3} (L^{3, \infty} \cap L^\infty).
$$
In particular, $\phi \in \mc E$ is in $L^1$ if and only if $\langle V \phi, y_k \rangle = 0$ and $\langle V \phi, y_k y_\ell \rangle = 0$ for $1 \leq k, \ell \leq 3$.

Let $\mc E_1 := \mc E \cap L^1$. Then $\codim_{\mc E} \mc E_1 \leq 12$.
\end{lemma}
\begin{proof}[Proof of Lemma \ref{decay_lemma}]
We start from the eigenfunction equation
$$
\phi(x) = -\frac 1 {4\pi} \int_{\R^3} \frac 1 {|x-y|} V(y) \phi(y) \dd y.
$$
Recall that $\langle \phi, V \rangle = 0$. Using (\ref{inequality}) we obtain that
$$
\Big|\phi(x) - \sum_{k=1}^3 \langle V \phi, y_k \rangle \frac {x_k}{|x|^3}\Big| \les \frac 1 {|x|^2} \int_{\R^3} \frac{|y|^2|V(y)||\phi(y)|\dd y}{|x-y|}.
$$
Since $\phi \in \langle x \rangle^{-2} L^\infty$ and $V \in L^1 \cap L^{3/2, 1}$, the right-hand side is in $|x|^{-2} (L^{3, \infty} \cap L^\infty)$.

Using (\ref{inequality2}) we obtain instead that
$$
\Big|\phi(x) - \sum_{k=1}^3 \langle \phi V, y_k \rangle \frac {x_k}{|x|^3} - \sum_{k, \ell=1}^3 \langle \phi V, y_k y_\ell \rangle \Big(\frac {\delta_{k\ell}}{2|x|^3} - \frac {3 x_k x_\ell}{2|x|^5} \Big)\Big| \les \frac 1 {|x|^3} \int_{\R^3} \frac{|y|^3|V(y)||\phi(y)|\dd y}{|x-y|}.
$$
Since $\phi \in \langle x \rangle^{-2} L^\infty$ and $V \in \langle x \rangle^{-1} L^1 \cap L^{3/2, 1}$, the right-hand side is in $|x|^{-3} (L^{3, \infty} \cap L^\infty)$.

These estimates matter only in the region $\{x:|x| \geq 1\}$, since near zero $\phi \in L^\infty \subset L^1(\{|x| \leq 1\})$. As $|x|^{-3} L^{3, \infty} \subset L^1(\{|x| \geq 1\})$ and $\frac {x_k}{|x|^3}, \frac {\delta_{k\ell}}{2|x|^3} - \frac {3 x_k x_\ell}{2|x|^5} \not \in L^1$ are linearly independent, it follows that $\phi \in \mc E$ is in $L^1$ if and only if all the coefficients $\langle V\phi, y_k \rangle$ and $\langle V\phi, y_k y_\ell \rangle$ are zero.

Then $\mc E_1$ is the kernel of a rank-$12$ map $\phi \mapsto (\langle \phi V, y_k \rangle, \langle \phi V, y_k y_\ell \rangle)$ from $\mc E$ to $\C^{12}$, so $\codim_{\mc E} \mc E_1 \leq 12$.
\end{proof}

\subsection{Wiener spaces}

\begin{definition}
For a Banach lattice $X$, let the space $\V_X$ consist of kernels $T(x, y, \sigma)$ such that, for each pair $(x, y)$, $T(x, y, \sigma)$ is a finite measure in $\sigma$ on $\R$ and
$$
M(T)(x, y):=\int_{\R} \dd |T(x, y, \sigma)|
$$
is an $X$-bounded operator.
\end{definition}

$\V_X$ is an algebra under
$$
(T_1 \ast T_2) (x, z, \sigma) := \int T_1(x, y, \rho) T_2(y, z, \sigma-\rho) \dd y \dd s.
$$
Elements of $\V_X$ have Fourier transforms
$$
\widehat T(x, y, \lambda) := \int_{\R} e^{-i\sigma\lambda} \dd T(x, y, \sigma)
$$
which are uniformly $X$-bounded operators, $\widehat T(\lambda) \in L^\infty_\lambda \B(X)$, and, for every $\lambda \in \R$, $\widehat T_1(\lambda) \widehat T_2(\lambda) = (T_1 \ast T_2)^{\wedge}(\lambda)$.

The space $\V_X$ contains elements of the form $\delta_0(\sigma) T(x, y)$, whose Fourier transform is constantly the operator $T(x, y) \in \B(X)$. In particular, rank-one operators $\delta_0(\sigma) \phi(x) \otimes \psi(y)$ are in $\V_X$ when $\psi \in X^*$, $\phi \in X$. More generally, $f(\sigma) T(x, y) \in \V_X$ if $f \in L^1$ and $T \in \B(X)$.

Moreover, for two Banach lattices $X$ and $Y$ of functions on $\R^3$, we also define the space $\V_{X, Y}$ of kernels $T(x, y, \sigma)$ such that $M(T)(x, y)$ is a bounded operator from $X$ to $Y$. The category of such operators forms an algebroid, in the sense that
$$
\|T_1 \ast T_2\|_{\V_{X, Z}} \leq \|T_1\|_{\V_{Y, Z}} \|T_2\|_{\V{X, Y}}.
$$

For example, note that $(R_{0}((\lambda+i0)^2))^\wedge \in \V_{L^{3/2, 1}, L^\infty} \cap \V_{L^1, L^{3, \infty}}$ and $(\partial_{\lambda} R_{0}((\lambda+i0)^2))^\wedge \in \V_{L^1, L^\infty}$. Indeed, the Fourier transform in $\lambda$ is
$$
(R_{0}((\lambda+i0)^2))^\wedge(\sigma)(x, y) = (4\pi \sigma)^{-1} \delta_{|x-y|}(\sigma),
$$
so $\ds M((R_{0}((\lambda+i0)^2))^\wedge) = \frac 1 {4\pi|x-y|}$. Clearly $\ds \frac 1 {4\pi|x-y|}$ is in $\B(L^{3/2, 1}, L^\infty) \cap \B(L^1, L^{3, \infty})$.

Likewise, $(\partial_\lambda R_{0}((\lambda+i0)^2))^\wedge(\sigma)(x, y) = (4\pi)^{-1} i  \delta_{|x-y|}(\sigma)$, so $\ds M((\partial_\lambda R_0((\lambda+i0)^2))^\wedge) = (4\pi)^{-1} 1 \otimes 1$, which is in $\B(L^1, L^\infty)$.

A space that will repeatedly intervene in computations is
\begin{definition}
$\W=\{L \mid L^\vee \in \V_{L^{3/2, 2}} \cap \V_{L^{3, 2}},\ (\partial_{\lambda} L)^\vee \in \V_{L^{3/2, 2}, L^{3, 2}}\}$.
\end{definition}
This space has the algebra property that $L_1, L_2 \in \W \implies L_1(\lambda) L_2(\lambda) \in~\W$.

The following technical lemma will be useful:
\begin{lemma}[Fourier transforms]\lb{fourier}
$$\begin{aligned}
M\Big(\Big(\frac {e^{is|x-y|}}{4\pi|x-y|}\Big)^\wedge\Big) &= \frac 1 {4\pi|x-y|};\\
M\Big(\Big(\partial_s \frac {e^{is|x-y|}}{4\pi|x-y|}\Big)^\wedge\Big) &= \frac {1 \otimes 1}{4\pi};\\
M\Big(\Big(\frac{R_0((s+i0)^2)-R_0(0)}{s}\Big)^\wedge\Big) &= \frac {1 \otimes 1} {4\pi};\\
M\Big(\Big(\partial_s \frac{R_0((s+i0)^2)-R_0(0)}{s}\Big)^\wedge\Big) &= \frac {|x-y|} {8\pi};\\
M\Big(\displaystyle \Big(\frac{R_0((s+i0)^2)-R_0(0)-is \frac{1 \otimes 1}{4\pi}}{s^2}\Big)^\wedge\Big) &= \frac {|x-y|} {8\pi};\\
M\Big(\Big(\partial_s \displaystyle \frac{R_0((s+i0)^2)-R_0(0)-is \frac{1 \otimes 1}{4\pi}}{s^2}\Big)^\wedge\Big) &= \frac{|x-y|^2} {24\pi}.
\end{aligned}
$$
\end{lemma}
\begin{proof}
Let $a>0$. Observe that the Fourier transform of $e^{i\lambda a}$ in $\lambda$ is $\delta_a(t)$. Then
$$
\frac {e^{i\lambda a}-1} {i\lambda} = \int_0^a e^{i\lambda b} \dd b,
$$
so $\Big(\displaystyle \frac {e^{i\lambda a}-1} {i\lambda}\Big)^{\wedge}=\chi_{[0, a]}(\lambda)$. Also
$$
\frac {e^{i\lambda a} - 1 - i\lambda a}{i\lambda^2} = \int_0^a \frac {e^{i\lambda b}-1}\lambda \dd b,
$$
so $\displaystyle\Big(\frac {e^{i\lambda a} - 1 -i\lambda a}{i\lambda^2}\Big)^{\wedge}=(a-t)\chi_{[0, a]}(t)$.

Note that $\displaystyle R_0((s+i0)^2) = \frac {e^{is|x-y|}}{4\pi|x-y|}$ has the Fourier transform $\frac {\delta_{|x-y|}(\sigma)}{4\pi|x-y|}$. Thus
$$
R_0((s+i0)^2)^\wedge = \Big(\frac {e^{is|x-y|}}{4\pi|x-y|}\Big)^\wedge = \frac {\delta_{|x-y|}(\sigma)}{4\pi|x-y|}.
$$
Integrating the absolute value in $\sigma$ we obtain $\ds\frac 1 {4\pi|x-y|}$.

Likewise, $\displaystyle \Big(\frac{R_0((s+i0)^2)-R_0(0)}{s}\Big)^{\wedge} = \displaystyle \frac{i\chi_{[0, |x-y|]}(\sigma)}{4\pi|x-y|}$. Integrating the absolute value in $\sigma$ we get $\ds\frac 1 {4\pi} =\frac{1 \otimes 1}{4\pi}$.

The Fourier transform of the derivative is $\displaystyle \Big(\partial_s \frac{R_0((s+i0)^2)-R_0(0)}{s}\Big)^{\wedge} = \displaystyle \frac{i\sigma \chi_{[0, |x-y|]}(\sigma)}{4\pi|x-y|}$. Integrating in $\sigma$ we obtain $\ds \frac{|x-y|}{8\pi}$.

Next,
\be\begin{aligned}\lb{dif2}
\Big(\displaystyle \frac{R_0((s+i0)^2)-R_0(0)-is 1 \otimes 1}{s^2}\Big)^{\wedge} &= \Big(\frac {e^{is|x-y|}-1-is|x-y|}{4\pi s^2|x-y|}\Big)^{\wedge}\\
&= \displaystyle \frac {(|x-y|-\sigma)\chi_{[0, |x-y|]}(\sigma)}{4\pi |x-y|}.
\end{aligned}\ee
Integrating in $\sigma$ we obtain $\ds \frac{|x-y|}{8\pi}$.

The Fourier transform of the derivative is
$$
\Big(\partial_s \displaystyle \frac{R_0((s+i0)^2)-R_0(0)-is 1 \otimes 1}{s^2}\Big)^{\wedge} =\displaystyle \frac {\sigma(|x-y|-\sigma)\chi_{[0, |x-y|]}(\sigma)}{4\pi|x-y|}.
$$
Integrating in $\sigma$ we obtain $\ds \frac {|x-y|^2}{24\pi}$.

\end{proof}

\subsection{Regular points and regular Hamiltonians}
Before examining the possible singularity at zero, we study what happens at regular points in the spectrum.

Recall the notation $V_1=|V|^{1/2}$ and $V_2=|V|^{1/2} \sgn V$. The following two properties play an important part in the study:
\begin{lemma}\lb{lemma_22} Let $T(x, y, \rho) := \frac {V_2(x) V_1(y)}{4\pi|x-y|}\delta_{-|x-y|}(\rho)$, so $\widehat T(\lambda) = V_2 R_0((\lambda+~i0)^2) V_1$.
\begin{list}{\labelitemi}{\leftmargin=1em}
\item[C1.] $\lim_{R\to \infty} \|\chi_{\rho\geq R}(\rho) T(\rho)\|_{\V_{L^{3/2, 2}} \cap \V_{L^{3, 2}}} =0$.
\item[C2.] For some $N \geq 1$ $\lim_{\epsilon \to 0} \|T^N(\rho+\epsilon) - T^N(\rho)\|_{\V_{L^{3/2, 2}} \cap \V_{L^{3, 2}}} = 0$.
\end{list}
\end{lemma}
Here the powers of $T$ mean repeated convolution. We refer the reader to similar properties that appear in the proof of Theorem 5 in \cite{becgol}. 

\begin{proof}[Proof of Lemma \ref{lemma_22}]
Suppose $V_1$ and $V_2$ are bounded functions with compact support in $B(0, D)$.  It follows that for $R > 2D$ $\chi(t/R) T(t) = 0$, so in particular $\|\chi_{t \geq R}T\|_{\V_{L^{3/2, 2}} \cap \V_{L^{3, 2}}} \to 0$ as $R \to \infty$, and property C1 is preserved by taking the limit of $V_1$ and $V_2$ in~$L^{3, 2}$.

Next, fix $p \in (1, 4/3]$ and assume that $V_1$ and $V_2$ are bounded and of compact support.

Since $V_1$ and $V_2$ are bounded and of compact support, $T$ also has the local and distal properties
$$
\lim_{\epsilon \to 0} \Big\|\chi_{<\epsilon}(|x-y|) \frac{V_2(x) V_1(y)}{|x-y|}\Big\|_{\B(L^{3/2, 2}) \cap \B(L^{3, 2})} = 0
$$
and
$$
\lim_{R \to \infty} \Big\|\chi_{>R}(|x-y|) \frac{V_2(x) V_1(y)}{|x-y|}\Big\|_{\B(L^{3/2, 2}) \cap \B(L^{3, 2})} = 0.
$$
Combined with condition C1, this implies that for any $\epsilon>0$ there exists a cutoff function $\chi$ compactly supported in $(0, \infty)$ such that
$$
\|\chi(\rho) T(\rho) - T(\rho)\|_{\V_{L^{3/2, 2}} \cap \V_{L^{3, 2}}} < \epsilon.
$$
Thus, it suffices to show that condition C2 holds for $\chi(\rho) T(\rho)$.

The Fourier transform of $\chi(\rho) T(\rho)$ has the form
\be\lb{ft}
(\chi(\rho) T(\rho))^{\wedge}(\lambda) = V_2(x) \frac {e^{i\lambda|x-y|}}{4\pi|x-y|} \chi(|x-y|) V_1(y).
\ee
Such oscillating kernels have decay in the $L^p$ operator norm for $p>1$. By the Lemma of \cite{stein}, page 392, with $p'$ being the dual exponent $\frac 1 p + \frac 1 {p'} = 1$,
\be\lb{stein}
\|(\chi(\rho) T(\rho))^{\wedge}(\lambda) f\|_{L^p} \les \lambda^{-3/p'} \|f\|_{L^p}.
\ee
Taking into account the fact that $(\chi(\rho) T(\rho))^{\wedge}(\lambda)$ has a kernel bounded in absolute value by $\ds \frac {|V(x)|^{1/2}|V(y)|^{1/2}}{4\pi|x-y|}$ (where $|V|^{1/2}=V_1$ is bounded and has compact support by assumption), it follows that $(\chi(\rho) T(\rho))^{\wedge}(\lambda)$ is uniformly bounded in $\B(X, L^p)$, $\B(L^p, X)$, and $\B(L^p)$ for all $\lambda$, where $X$ is $L^{3/2, 2}$ or $L^{3, 2}$. Therefore, by also using (\ref{stein}) for the middle factors,
$$
\|\big((\chi(\rho) T(\rho))^{\wedge}(\lambda)\big)^N f\|_{X} \les \langle\lambda \rangle^{-3(N-2)/p'} \|f\|_{X}.
$$
For $N>2+2p'/3$, this shows that $\partial_{\rho} (\chi(\rho) T(\rho))^N$ are uniformly bounded operators in $\B(X)$, where $X$ is either $L^{3/2, 2}$ or $L^{3, 2}$. Since $(\chi(\rho) T(\rho))^N$ has compact support in $\rho$, this in turn implies C2.

For general $V \in L^{3/2, 1}$, choose a sequence of bounded compactly supported approximations for which C2 holds, as shown above. By a limiting process, we obtain that C2 also holds for $V$.
\end{proof}

\begin{lemma}\lb{lemma23}
Let $\widehat T(\lambda) = V_2 R_0((\lambda+i0)^2) V_1$. Assume that $V \in L^{3/2, 1}$ and let $\lambda_0 \ne 0$. Consider a cutoff function $\chi$. Then, for $\epsilon \ll 1$, $\chi(\frac{\lambda-\lambda_0} \epsilon) (I + \widehat T(\lambda)^{-1}) \in \W$. 
The same holds for $\lambda_0=0$ if $V$ is a generic potential.

Infinity has the same property: for $R \gg 1$, $(1-\chi(\lambda/R)) (I+\widehat T(\lambda))^{-1} \in~\W$.
\end{lemma}
\begin{proof}[Proof of Lemma \ref{lemma23}]
Note that $I+\widehat T(\lambda_0)$ is invertible in $\B(L^{3/2, 2})$ and in $\B(L^{3, 2})$ for all $\lambda_0 \ne 0$, the only issue being at zero.

Indeed, assume that $I+\widehat T(\lambda_0)$ is not invertible in $\B(L^{3/2, 2})$; then by Fredholm's alternative there exists a nonzero $f \in L^{3/2, 2}$ such that
$$
f=-V_2 R_0((\lambda_0+i0)^2) V_1 f.
$$
Let $V_1=V_1^1 + V_1^2$, $V_2=V_2^1+V_2^2$, where $V_1^1$ and $V_2^1$ have compact support and are bounded and $\|V_1^2\|_{L^{3, 2}}$, $\|V_2^2\|_{L^{3, 2}} \ll1$. Then
$$
f = -(I + V_2 R_0((\lambda_0+i0)^2) V_1^2 + V_2^2 R_0((\lambda_0+i0)^2) V_1^1)^{-1} V_2^1 R_0((\lambda_0+i0)^2) V_1^1 f,
$$
which implies that $f \in L^2$. Letting $g=R_0((\lambda_0+i0)^2) V_1 f$, we obtain a nonzero $L^{6, \infty}$ solution $g$ of the equation
$$
g=-R_0((\lambda_0+i0)^2) V g.
$$
However, this is impossible for $\lambda_0 \ne 0$ due to the results of Ionescu--Jerison \cite{ionjer} and Goldberg--Schlag \cite{golsch}.

When $\lambda_0=0$, $g$ is a zero energy eigenfunction or resonance for $H=-\Delta+V$, which cannot happen if $V$ is a generic potential.

Let $S_\epsilon(\lambda) = \chi(\frac{\lambda-\lambda_0} \epsilon)(\widehat T(\lambda)-\widehat T(\lambda_0))$. A simple argument based on condition C1 shows that $\lim_{\epsilon \to 0} \|S_\epsilon^\vee\|_{\V_{L^{3/2, 2}} \cap \V_{L^{3, 2}}}=0$. Then
$$\begin{aligned}
\chi(\frac{\lambda-\lambda_0}\epsilon) (I+\widehat T(\lambda))^{-1} &= \chi(\lambda/\epsilon) \big(I+\widehat T(\lambda_0) + \chi(\frac{\lambda-\lambda_0}{2\epsilon})(\widehat T(\lambda) - \widehat T(\lambda_0))\big)^{-1}\\
&=\chi(\frac{\lambda-\lambda_0} \epsilon) (I+\widehat T(\lambda_0))^{-1} \sum_{k=0}^\infty (-1)^k (S_{2\epsilon}(\lambda) (I+\widehat T(\lambda_0))^{-1})^k.
\end{aligned}$$
The Fourier transform of the series above converges for sufficiently small $\epsilon$, showing that $(\chi(\frac{\lambda-\lambda_0}\epsilon) (I + \widehat T(\lambda))^{-1})^\vee \in \V_{L^{3/2, 2}} \cap V_{L^{3, 2}}$.

Concerning the derivative,
$$
\chi(\frac{\lambda-\lambda_0}\epsilon) \partial_\lambda (I+\widehat T(\lambda))^{-1} = - \chi(\frac{\lambda-\lambda_0}\epsilon) (I+\widehat T(\lambda))^{-1} \partial_{\lambda} \widehat T(\lambda) \chi(\frac{\lambda-\lambda_0}{2\epsilon}) (I+\widehat T(\lambda))^{-1}.
$$
Here $(\chi(\frac{\lambda-\lambda_0}{2\epsilon}) (I + \widehat T(\lambda))^{-1})^\vee \in \V_{L^{3/2, 2}} \cap \V_{L^{3, 2}}$ and $(\partial_{\lambda} \widehat T(\lambda))^\vee \in \V_{L^{3/2, 2}, L^{3, 2}}$ since $\ds M((\partial_{\lambda} T(\lambda))^\vee) = \frac{|V_2(x)| \otimes |V_1(y)|}{4\pi}$. Then $(\chi(\frac {\lambda - \lambda_0}\epsilon) \partial_\lambda(I + \widehat T(\lambda))^{-1})^\vee \in \V_{L^{3/2, 2}, L^{3, 2}}$.

The term $\partial_\lambda \chi(\frac{\lambda-\lambda_0}\epsilon) (I+\widehat T(\lambda))^{-1}$ is handled as above. Thus $(\partial_\lambda (\chi(\frac {\lambda - \lambda_0}\epsilon) (I + \widehat T(\lambda))^{-1}))^\vee \in \V_{L^{3/2, 2}, L^{3, 2}}$.

At infinity, for any real number $L$ one can express $(1 - \chi(\lambda/R)) \widehat T(\lambda)$ as the Fourier transform of
\begin{equation*}
S_R(\rho) = \big(T - R\widecheck{\chi}(R\,\cdot\,) * T\big)(\rho) = 
\int_\R R \widecheck{\chi}(R\sigma) [T(\rho) - T(\rho-\sigma)]\,d\sigma
\end{equation*}
Thanks to condition C2, the norm of the right-hand side integral vanishes as $L \to \infty$.  This makes it possible to construct an inverse Fourier transform for
\begin{equation*}
(1 - \chi(\lambda/R))\big(I + \widehat T(\lambda)\big)^{-1} 
= (1 - \chi(\lambda/R)) 
 \sum_{k=0}^\infty (-1)^k \Big(\big(1 - \chi(2\lambda/R))\widehat T(\lambda)\Big)^k
\end{equation*}
via this power series expansion, which converges for sufficiently large $R$.

If only $T^N$ satisfies C2 then one constructs an
inverse Fourier transform for $(1-\chi(\lambda/R))
 (I - (-\widehat T)^N(\lambda))^{-1}$ in this manner and observes that
\begin{equation*}
(1 - \chi(\lambda/R))\big(I + \widehat T(\lambda)\big)^{-1}
  = (1 - \chi(\lambda/R))\big(I - (-\widehat T(\lambda))^N \big)^{-1}
\sum_{k=0}^{N-1}(-1)^k \widehat T^k(\lambda).
\end{equation*}

Finally, concerning the derivative in a neighborhood of infinity, we note that
$$
(1-\chi(\lambda/R)) \partial_\lambda (I + \widehat T(\lambda))^{-1} = -(1-\chi(\lambda/R)) (I + \widehat T(\lambda))^{-1} \partial_\lambda \widehat T(\lambda) (1-\chi(2\lambda/R)) (I + \widehat T(\lambda))^{-1}.
$$
Here $((1-\chi(2\lambda/R)) (I + \widehat T(\lambda))^{-1})^\vee \in \V_{L^{3/2, 2}} \cap \V_{L^{3, 2}}$ and $(\partial_\lambda \widehat T(\lambda))^\vee \in \V_{L^{3/2, 2}, L^{3, 2}}$. Therefore $((1-\chi(\lambda/R)) \partial_\lambda (I + \widehat T(\lambda))^{-1})^\vee \in \V_{L^{3/2, 2}, L^{3, 2}}$ and furthermore $(\partial_\lambda((1-\chi(\lambda/R)) (I + \widehat T(\lambda))^{-1}))^\vee \in \V_{L^{3/2, 2}, L^{3, 2}}$.
\end{proof}

In the case when $H$ is generic, we can cover the whole spectrum $[0, \infty)$ by open neighborhoods of regular points, plus an open neighborhood of infinity, and choose a subordinate partition of unity. We retrieve a form of Theorem 2 of \cite{becgol}:
\begin{theorem}
\label{thm:dispersive}
Let $V \in L^{3/2, 1}$ be a real-valued potential for which the
Schr\"odinger operator $H = -\Delta + V$ has no resonances or eigenvalues
at zero energy.  Then
\begin{equation} \label{eq:dispersive}
\big\|e^{-itH} P_c f\big\|_{\infty} 
  \les |t|^{-\frac32}\|f\|_1. 
\end{equation}
\end{theorem}
In the context of the wave equation, again if the Hamiltonian $H$ is generic, we retrieve the results of \cite{becgol2}.

\begin{proof}[Proof of Theorem \ref{thm:dispersive}]
Consider a sufficiently large $R$ such that $(1-\chi(\lambda/R))(I+\widehat T(\lambda))^{-1} \in \W$, by Lemma \ref{lemma23}. Also by Lemma \ref{lemma23}, for every $\lambda_0 \in [-4R, 4R]$ (including zero, since $V$ is a generic potential) there exists $\epsilon(\lambda_0)>0$ such that $\chi(\frac{\lambda-\lambda_0}{\epsilon(\lambda_0)}) (I+\widehat T(\lambda))^{-1} \in \W$.

Since $[-4R, 4R]$ is a compact set, there exists a finite covering $[-4R, 4R] \subset \bigcup_{k=1}^N (\lambda_k-\epsilon(\lambda_k), \lambda_k+\epsilon(\lambda_k))$. Then we construct a finite partition of unity on $\R$ by smooth functions $1=\sum_{k=1}^N \chi_k(\lambda) + \chi_\infty(\lambda)$, where $\supp \chi_\infty \subset \R \setminus (-2R, 2R)$ and $\supp \chi_k \subset [\lambda_k-\epsilon(\lambda_k), \lambda_k+\epsilon(\lambda_k)]$. By our construction, for each $1 \leq k \leq N$ and for $k=\infty$, $\chi_k(\lambda) (I+\widehat T(\lambda))^{-1} \in \W$, so summing up we obtain that $(I+\widehat T(\lambda))^{-1} \in \W$.

By spectral calculus, we express the perturbed evolution as
\be\lb{explicit_computation}\begin{aligned}
&e^{itH} P_c f = \frac 1 {2\pi i} \int_0^\infty e^{it\lambda} (R_V(\lambda+i0)-R_V(\lambda-i0)) f \dd \lambda \\
&= \frac 1 {\pi i} \int_{-\infty}^\infty e^{it\lambda^2} R_V((\lambda+i0)^2) f \lambda \dd \lambda \\
&= \frac 1 {\pi i} \int_{-\infty}^\infty e^{it\lambda^2} (R_0((\lambda+i0)^2) - R_0((\lambda+i0)^2) V_1 (I+\widehat T(\lambda))^{-1} V_2 R_0((\lambda+i0)^2)) f \lambda \dd \lambda \\
&= \frac 1 {2 \pi t} \int_{-\infty}^\infty e^{it\lambda^2} \partial_\lambda (R_0((\lambda+i0)^2) - R_0((\lambda+i0)^2) V_1 (I+\widehat T(\lambda))^{-1} V_2 R_0((\lambda+i0)^2)) f \dd \lambda \\
&= \frac C {t^{3/2}} \int_{-\infty}^\infty e^{i\frac{\rho^2}{4t}} (\partial_\lambda (R_0((\lambda+i0)^2) - R_0((\lambda+i0)^2) V_1 (I+\widehat T(\lambda))^{-1} V_2 R_0((\lambda+i0)^2)))^\vee(\rho) f \dd \rho.
\end{aligned}\ee

Since $(I+\widehat T(\lambda))^{-1} \in \W$, it follows that $(\partial_\lambda (I+\widehat T(\lambda))^{-1})^\vee \in \V_{L^{3/2, 2}, L^{3, 2}}$. Taking into account that $R_0((\lambda+i0)^2) V_1 \in \V_{L^1, L^{3/2, 2}}$ and $V_2 R_0((\lambda+i0)^2) \in \V_{L^{3, 2}, L^\infty}$, we obtain that
$$
R_0((\lambda+i0)^2) V_1 (I+\widehat T(\lambda))^{-1} V_2 R_0((\lambda+i0)^2) \in \V_{L^1, L^\infty}.
$$
By definition, this ensures a bound of $|t|^{-3/2}$ for this expression's contribution to (\ref{explicit_computation}). The other terms are handled similarly.
\end{proof}

We next consider the effect of singularities at zero.

\subsection{Exceptional Hamiltonians of the first kind}
Let
$$
Q = - \frac 1 {2\pi i} \int_{|z+1|=\delta} (V_2 R_0(0) V_1 - zI)^{-1} \dd z
$$
and $\ov Q = 1-Q$. Assuming that $H=-\Delta+V$ has only a resonance $\phi$ at zero, then (recalling that $-\langle \phi, V \phi \rangle = 1$) by the analytic Fredholm theorem
$$
Q = -V_2 \phi \otimes V_1 \phi.
$$

The resonance $\phi \in \mc M$ satisfies the equation $\phi = - R_0(0) V \phi$. Since $\phi \in L^{3, \infty} \cap L^\infty$, $Q$ is bounded on $L^{3/2, 2}$ and on $L^{3, 2}$, so the constant family of operators $Q$ is in $\W$. Moreover, $Q$ is in $\B(L^{3/2, 2}, L^{3, 2})$ and in $\B(L^{3, 2}, L^{3/2, 2})$.

Note that, since
$$
e^{i\lambda|x-y|}-1 \les \min(1, \lambda |x-y|) \implies e^{i\lambda|x-y|}-1 \les \lambda^{\delta} |x-y|^\delta,
$$
one has
\be\lb{holder}
V_2(x) \Big(\frac {e^{i\lambda|x-y|}}{|x-y|} - \frac 1 {|x-y|}\Big) V_1(y) \les |V_2(x)| \lambda |V_1(y)|.
\ee
Thus,
when $V \in \langle x \rangle^{-1} L^{3/2, 1}$, $I+\widehat T(\lambda) = I + V_2 R_0((\lambda+i0)^2) V_1$ is Lipschitz continuous in $\B(L^2)$. This implies that, more generally, when $V \in L^{3/2, 1}$ $\widehat T(\lambda)$ is continuous in $\B(L^2)$ (the proof is by approximation).

In a similar manner, by approximating $V \in L^{3/2, 1}$ with $\langle x \rangle^{-2} L^{3/2, 1}$ potentials, we obtain that $\widehat T(\lambda)$ is continuous in $\B(L^{3/2, 2}) \cap \B(L^{3, 2})$.

Let
$$
K = (I + V_2 R_0(0) V_1 + Q)^{-1} \ov Q.
$$
Then $K$ is the inverse of $\ov Q (I + \widehat T(0)) \ov Q = \ov Q (I + V_2 R_0(0) V_1) \ov Q$ in $\B(\ov Q L^{3/2, 2} \cap \ov Q L^{3, 2})$, in the sense that
\be\lb{inverse}
K \ov Q (I + V_2 R_0(0)V_1) \ov Q = \ov Q (I+V_2R_0(0)V_1) \ov Q K = \ov Q.
\ee
By continuity $\ov Q (I+V_2 R_0((\lambda+i0)^2) V_1) \ov Q$ is also invertible for $|\lambda| \ll 1$.

The following lemma (Lemma 4.7 from Yajima \cite{yajima_disp}, also known as the Feshbach lemma) is extremely useful in studying the singularity at zero. 
\begin{lemma}\lb{lemma_invers} Let $X = X_0+X_1$ be a direct sum decomposition of a vector space $X$. Suppose that a linear operator $L \in \B(X)$ is written in the form 
$$
L = \bpm
L_{00}&L_{01}\\
L_{10}&L_{11}
\epm
$$
with respect to this decomposition and that $L^{-1}_{00}$ exists. Set 
$C = L_{11} - L_{10} L_{00}^{-1} L_{01}$. Then, $L^{-1}$ exists if and only if $C^{-1}$ exists. In this case
\be\lb{invers}
L^{-1}= \bpm L_{00}^{-1} + L_{00}^{-1} L_{01} C^{-1} L_{10} L_{00}^{-1}&-L_{00}^{-1} L_{01} C^{-1} \\ 
-C^{-1} L_{10} L_{00}^{-1} & C^{-1}
\epm.
\ee
\end{lemma}

By definition, an exceptional point $\lambda \in \C$ is one where $I + V_2 R_0(\lambda) V_1$ is not $L^2$-invertible.

\begin{lemma}\lb{lemma24} Assume that $V \in \langle x \rangle^{-2} L^{3/2, 1} \subset \langle x \rangle^{-1} L^1 \cap L^{3/2, 1}$ and that $H=-\Delta+V$ is exceptional of the first type, with a resonance $\phi$ at zero.
Let $\chi$ be a fixed cutoff function. Then for some $\epsilon > 0$
$$
\chi(\lambda/\epsilon)(I + \widehat T(\lambda))^{-1} 
= L(\lambda) - \lambda^{-1} \chi(\lambda/\epsilon) \frac{4\pi i}{|\langle V, \phi\rangle|^2} V_2 \phi \otimes V_1 \phi,
$$
where $L \in \W$.

Moreover, zero is an isolated exceptional point, so $H=-\Delta+V$ has finitely many negative eigenvalues.

\end{lemma}
The computations in the proof of this lemma parallel those in Section 4.3 of \cite{yajima_disp}. The main difference is using $\widehat L^1$-related spaces instead of H\"{o}lder spaces.
\begin{proof}[Proof of Lemma \ref{lemma24}]
We apply Lemma \ref{lemma_invers} to
$$\begin{aligned}
I + \widehat T(\lambda) &:= 
\bpm \ov Q (I +\widehat T(\lambda)) \ov Q & \ov Q\widehat T(\lambda)Q \\ Q\widehat T(\lambda)\ov Q & Q (I+\widehat T(\lambda)) Q \epm = \bpm T_{00}(\lambda) & T_{01}(\lambda) \\ T_{10} (\lambda) & T_{11} (\lambda)\epm.
\end{aligned}$$
Note that $T_{00}(\lambda):=\ov Q(I+V_2R_0((\lambda+i0)^2)V_1) \ov Q$ is invertible in $\B(\ov Q L^{3/2, 2}) \cap \B(\ov Q L^{3, 2})$ for $|\lambda|\ll 1$, because
$$
T_{00}(0) = \ov Q(I + \widehat T(0)) \ov Q = \ov Q(I + V_2R_0(0)V_1)\ov Q
$$
is invertible on $\ov Q L^{3/2, 2}$ and on $\ov Q L^{3, 2}$ of inverse $K$, see (\ref{inverse}), and $T_{00}(\lambda)$ is continuous in the norm of $\B(\ov Q L^{3/2, 2}) \cap \B(\ov Q L^{3, 2})$, see (\ref{holder}) above.


Furthermore, start from $(R_{0}((\lambda+i0)^2))^\wedge \in \V_{L^{3/2, 1}, L^\infty} \cap \V_{L^1, L^{3, \infty}}$ and $(\partial_{\lambda} R_{0}((\lambda+i0)^2))^\wedge \in \V_{L^1, L^\infty}$. We know that $|V|^{1/2} \in \B(L^{3/2, 2}, L^1) \cap \B(L^{\infty}, L^{3, 2}) \cap \B(L^{3, \infty}, L^{3/2, 2}) \cap \B(L^{3, 2}, L^{3/2, 1})$. Thus $V_2 R_0((\lambda+i0)^2) V_1 \in \W$ and $\ov Q$ preserves that. Then $T_{00}(\lambda) \in \W$ as well.

Next, since $T_{00}(0)$ is invertible, for small $\epsilon$ $\chi(\lambda/\epsilon) T_{00}^{-1}(\lambda) \in \W$. The proof is as follows: let $S_\epsilon(\lambda) := \chi(\lambda/\epsilon)\ov Q(\widehat T(\lambda) - \widehat T(0)) \ov Q$. A simple argument based on condition C1 shows that $\lim_{\epsilon \to 0} \|S_\epsilon^\vee\|_{\V_{L^{3/2, 2}}\cap \V_{L^{3, 2}}}=0$. Then
$$\begin{aligned}
\chi(\lambda/\epsilon) T_{00}^{-1}(\lambda) &= \chi(\lambda/\epsilon) \big(T_{00}(0) + \chi(\lambda/(2\epsilon)) \ov Q (\widehat T(\lambda) - \widehat T(0)) \ov Q\big)^{-1}\\
&=\chi(\lambda/\epsilon) T_{00}^{-1}(0) \sum_{k=0}^\infty (-1)^k (S_{2\epsilon}(\lambda) T_{00}^{-1}(0))^k.
\end{aligned}$$
The series above converges for sufficiently small $\epsilon$, showing that $\chi(\lambda/\epsilon) T_{00}^{-1}(\lambda) \in \V_{L^{3/2, 2}} \cap \V_{L^{3, 2}}$.

Concerning the derivative,
$$
\chi(\lambda/\epsilon) \partial_\lambda T_{00}^{-1}(\lambda) = - \chi(\lambda/\epsilon) T_{00}^{-1}(\lambda) \partial_{\lambda} T_{00}(\lambda) \chi(\lambda/(2\epsilon)) T_{00}^{-1}(\lambda).
$$
In this expression $(\chi(\lambda/(2\epsilon)) T_{00}^{-1}(\lambda))^\vee \in \V_{L^{3/2, 2}} \cap \V_{L^{3, 2}}$ and $(\partial_{\lambda} T_{00}(\lambda))^\vee \in \V_{L^{3/2, 2}, L^{3, 2}}$. Thus
$(\chi(\lambda/\epsilon) \partial_\lambda T_{00}^{-1}(\lambda))^\vee \in \V_{L^{3/2, 2}, L^{3, 2}}$.

This computation shows that $\chi(\lambda/\epsilon) T_{00}^{-1}(\lambda) \in \W$.

Let
$$\begin{aligned}
J (\lambda) &:= \frac{\widehat T(\lambda) - (V_2 R_0(0) V_1 + i \lambda (4\pi)^{-1} V_2 \otimes V_1 )}{\lambda^2} \\
&= \frac{V_2 R_0((\lambda+i0)^2) V_1 - V_2 R_0(0) V_1 - i \lambda (4\pi)^{-1} V_2 \otimes V_1}{\lambda^2}.
\end{aligned}$$

Then (recall that $Q=-V_2 \phi \otimes V_1 \phi$)
\be\lb{calcul}\begin{aligned}
T_{11}(\lambda) &= Q (I+\widehat T(\lambda)) Q = Q (I + V_2 R_0((\lambda+i0)^2) V_1) Q\\
&=Q (V_2 R_0((\lambda+i0)^2) V_1 - V_2 R_0(0) V_1) Q\\
&=V_2 \phi \otimes V\phi(R_0((\lambda+i0)^2) - R_0(0)) V\phi \otimes V_1\phi \\
&= \Big(\lambda \frac{|\langle V, \phi \rangle|^2}{4i\pi} - \lambda^2 \langle V_1 \phi, J (\lambda) V_2 \phi \rangle\Big) Q \\
&=(\lambda a^{-1} - \lambda^2 \langle V_1 \phi, J(\lambda) V_2 \phi) \rangle) Q\\
&:= \lambda c_0(\lambda) Q.
\end{aligned}\ee
Note that $c_0(0) = a^{-1} \ne 0$. Recall that $\ds a=\frac {4i\pi}{|\langle V, \phi\rangle|^2}$.

By the third line of (\ref{calcul}), $c_0(\lambda) \in \widehat L^1$ if
$$
\int_{\R^3} \int_{\R^3} V(x)\phi(x) V(y) \phi(y) \Big\|\frac{e^{i\lambda|x-y|} - 1}{\lambda |x-y|}\Big\|_{\widehat L^1_{\lambda}} \dd x \dd y < \infty.
$$
For every $x$ and $y$, by Lemma \ref{fourier}
$$
\Big\|\frac{e^{i\lambda|x-y|} - 1}{\lambda |x-y|}\Big\|_{\widehat L^1_{\lambda}}=\Big\|\frac {\chi_{[0, |x-y|]}(t)}{|x-y|}\Big\|_{L^1_t}=1,
$$
so it is enough to assume that $V \phi \in L^1$, i.e.\ that $V \in L^{3/2, 1}$, to prove that $c_0(\lambda) \in \widehat L^1$.

In order for $\partial_\lambda c_0(\lambda)$ to be in $\widehat L^1$, it suffices that
$$
\int_{\R^3} \int_{\R^3} V(x)\phi(x) V(y) \phi(y) \Big\|\partial_\lambda \frac{e^{i\lambda|x-y|} - 1}{\lambda |x-y|}\Big\|_{\widehat L^1_{\lambda}} \dd x \dd y < \infty.
$$
For every $x$ and $y$, by Lemma \ref{fourier},
$$
\Big\|\partial_\lambda\frac{e^{i\lambda|x-y|} - 1}{\lambda |x-y|}\Big\|_{\widehat L^1_{\lambda}}=\Big\|\frac {t\chi_{[0, |x-y|]}(t)}{|x-y|}\Big\|_{L^1_t}=\frac{|x-y|} 2,
$$
so $\partial_{\lambda} c_0(\lambda) \in \widehat L^1$ when $V\phi \in \langle x \rangle^{-1} L^1$, i.e.\ when $V \in L^1$.

Regarding $J(\lambda)$, when $V \in L^1$ then
\be\lb{j}
\langle J(\lambda) V_2 \phi, V_1 \phi \rangle = \Big\langle \frac{R_0((\lambda+i0)^2) - R_0(0) - i \lambda (4\pi)^{-1} 1 \otimes 1}{\lambda^2} V\phi, V \phi \Big\rangle \in \widehat L^1_\lambda.
\ee
Moreover, when $V \in  \langle x \rangle^{-1} L^1$, $\langle \partial_\lambda J(\lambda) V_2 \phi, V_1 \phi \rangle \in~\widehat L^1_\lambda$.

Furthermore, considering the fact that $\phi+R_0(0)V\phi=0$, let us define
$$\begin{aligned}
\lambda\tilde\psi(\lambda) &:= (I+\widehat T(\lambda))V_2 \phi = (V_2 R_0((\lambda+i0)^2)V-V_2 R_0(0)V)\phi\\
&= \lambda\Big(i \frac{V_2 \otimes V_1}{4\pi} + \lambda J (\lambda)\Big) V_2\phi
\end{aligned}$$
and
$$\begin{aligned}
\lambda \tilde \psi^*(\lambda) &:= (I+\widehat T(\lambda)^*) V_1 \phi = (V_1 R_0^*((\lambda+i0)^2) V - V_1 R_0(0)V) \phi \\
&= \lambda \Big(-i \frac{V_1 \otimes V_2}{4\pi} + \lambda J^*(\lambda)\Big) V_1 \phi.
\end{aligned}$$

Note that $\ds M(J(\lambda)^\vee) = |V_2(x)| \frac {|x-y|}{8\pi} |V_1(y)|$ is a bounded operator from $L^{3/2, 2}$ to $L^{3, 2}$, assuming that $V \in \langle x \rangle^{-2} L^{3/2, 1}$. Thus $J(\lambda)^\vee \in \V_{L^{3/2, 2}, L^{3, 2}}$ and same goes for $\lambda \partial_{\lambda} J(\lambda)$.

Moreover, $\ds M((\lambda J(\lambda))^\vee) = \frac {|V_2| \otimes |V_1|}{2\pi}$. Thus $(\lambda J(\lambda))^\vee \in \V_{L^2}$ for $V \in L^1$ and $(\lambda J(\lambda))^\vee \in \V_{L^{3/2, 2}} \cap \V_{L^{3, 2}}$ when $V \in \langle x \rangle^{-2} L^{3/2, 1}$. Further note that $(\partial_\lambda (\lambda J(\lambda)))^\vee = (J(\lambda) + \lambda \partial_\lambda J(\lambda))^\vee \in \V_{L^{3/2, 2}, L^{3, 2}}$. It follows that $\lambda J(\lambda) \in \W$.

Then (recalling that $Q = - V_2\phi \otimes V_1 \phi$)
$$\begin{aligned}
T_{01}(\lambda)  &:=\ov Q \widehat T(\lambda) Q = \ov Q (I+\widehat T(\lambda)) Q = (I+\widehat T(\lambda)) Q - Q (I+\widehat T(\lambda)) Q\\
&= -\lambda \tilde \psi(\lambda) \otimes V_1\phi - \lambda c_0(\lambda) Q\\
&= -\lambda (\tilde \psi (\lambda) + c_0(\lambda)V_2\phi) \otimes V_1 \phi.
\end{aligned}$$

Likewise,
$$\begin{aligned}
T_{10}(\lambda) &= -\lambda V_2\phi \otimes (\tilde \psi^*(\lambda) + \ov{c_0(\lambda)}V_1 \phi).
\end{aligned}$$
By our above computations it follows that $T_{01}(\lambda)  = \lambda E_1(\lambda)$ and $T_{10}(\lambda) = \lambda E_2(\lambda)$ with $E_1$, $E_2 \in \W$.

Then $-T_{10} (\lambda) T_{00}^{-1}(\lambda) T_{01} (\lambda) = \lambda^2 c_1(\lambda) Q$, where
\be\lb{c1}\begin{aligned}
c_1(\lambda) &:= -\big\langle \tilde \psi^*(\lambda) + \ov{c_0(\lambda)}V_1 \phi , T_{00}^{-1}(\lambda)( \tilde \psi (\lambda) + c_0 (\lambda)V_2\phi) \big\rangle \\
&= -\Big\langle \Big(-i \frac{V_1 \otimes V_2}{4\pi} + \lambda J^*(\lambda)\Big) V_1 \phi + \ov{c_0(\lambda)}V_1\phi,\\
&T_{00}^{-1}(\lambda)\Big(\Big(i \frac{V_2 \otimes V_1}{4\pi} + \lambda J (\lambda)\Big) V_2\phi + c_0(\lambda) V_2\phi\Big)\Big\rangle.
\end{aligned}\ee
For example, one of the terms in (\ref{c1}) has the form
\be\lb{term}
\langle \lambda J^*(\lambda) V_1 \phi, T_{00}^{-1}(\lambda) \lambda J(\lambda) V_2 \phi \rangle.
\ee

Since $\lambda J(\lambda) \in \W$ and $\chi(\lambda/\epsilon) T_{00}^{-1}(\lambda) \in \W$ and since $V_1 \phi$, $V_2\phi \in L^{3/2, 2} \cap L^{3, 2}$, it immediately follows that $\chi(\lambda/\epsilon)(\ref{term})$ is in $\widehat L^1$ and its derivative is also in~$\widehat L^1$.


We then recognize from formula (\ref{c1}) that, for a cutoff function $\chi$, $\chi(\lambda/\epsilon) c_1(\lambda) \in \widehat L^1$ 
and $\chi(\lambda/\epsilon) \partial_\lambda c_1(\lambda) \in \widehat L^1$ when $V \in \langle x \rangle^{-2} L^{3/2, 1}$. 

Let
$$
C(\lambda) := T_{11}(\lambda) - T_{10}(\lambda)T_{00}^{-1}(\lambda)T_{01} (\lambda).
$$
Then
$$\begin{aligned}
C(\lambda) 
&= (\lambda a^{-1} - \lambda^2 \langle V_1\phi, J(\lambda) V_2\phi \rangle + \lambda^2 c_1(\lambda)) Q:= \lambda a^{-1}Q + \lambda^2 c_2(\lambda)Q.
\end{aligned}$$
Thus $C(\lambda)/\lambda$ is invertible for $|\lambda|\ll 1$ and 
when $V \in \langle x \rangle^{-2} L^{3/2, 1}$ one has that
$$\begin{aligned}
C^{-1}(\lambda) &= \frac 1 {\lambda a^{-1} + \lambda^2 c_2(\lambda)} Q \\
&= \Big(\frac 1 {\lambda a^{-1}} + \frac 1 {\lambda a^{-1} + \lambda^2 c_2(\lambda)} - \frac 1 {\ds\lambda a^{-1}}\Big) Q\\
&= \Big(\frac {a}{\lambda} - \frac{c_2(\lambda)}{\ds\big(a^{-1} + \lambda c_2(\lambda)\big)a^{-1}}\Big)Q\\
&:=a\lambda^{-1}Q + E(\lambda).
\end{aligned}$$
By our computations, such as (\ref{j}), $\chi(\lambda/\epsilon) c_2(\lambda) \in \widehat L^1$ and $\chi(\lambda/\epsilon) \partial_\lambda c_2(\lambda) \in \widehat L^1$. Therefore for sufficiently small $\epsilon$, as $Q \in \B(L^{3/2, 2}) \cap B(L^{3, 2}) \cap \B(L^{3/2, 2}, L^{3, 2})$, it follows that $\chi(\lambda/\epsilon) E(\lambda) \in \W$.

The inverse of $I + \widehat T(\lambda)$ is then given for small $\lambda$ by formula (\ref{invers}):
$$
(I + \widehat T)^{-1}= \bpm T_{00}^{-1} + T_{00}^{-1} T_{01} C^{-1} T_{10} T_{00}^{-1}&-T_{00}^{-1} T_{01} C^{-1} \\ 
-C^{-1} T_{10} T_{00}^{-1} & C^{-1}
\epm.
$$
Three of the matrix elements belong to $\W$ when localized by $\chi(\lambda/\epsilon)$. Indeed, recall that $\chi(\lambda/\epsilon) T_{00}^{-1}(\lambda) \in \W$ and $T_{10}(\lambda) = \lambda E_1(\lambda)$ and $T_{01}(\lambda) = \lambda E_2(\lambda)$, while $C^{-1} = \lambda^{-1} E_3(\lambda)$, with $E_1, E_2, \chi(\lambda/\epsilon) E_3 \in \W$.

The fourth matrix element is $C^{-1}$ in the lower-right corner, which is the sum of the regular term $\chi(\lambda/\epsilon)E(\lambda) \in \W$ and the singular term
$$
a \lambda^{-1} \chi(\lambda/\epsilon) Q = -a\lambda^{-1} \chi(\lambda/\epsilon) V_2\phi \otimes V_1 \phi.
$$
As an aside, note that $\lambda^{-1} (1-\chi(\lambda/\epsilon)) \in \widehat L^1$ and same for its derivative. Thus we can also write the singular term as $a \lambda^{-1}Q$.

Further note that $(I + \widehat T)^{-1}$ is well-defined on a whole cut neighborhood of zero by formula (\ref{invers}) above. Thus zero is an isolated exceptional point, so there are finitely many negative eigenvalues.

%
%
\end{proof}

The next lemma shows what happens in case the potential has the critical rate of decay.

\begin{lemma}\lb{lemma27} Assume that $V \in L^{3/2, 1}$ and that $H=-\Delta+V$ is exceptional of the first kind. Let $\chi$ be a standard cutoff function. Then
$$
\chi(\lambda/\epsilon) (I + \widehat T(\lambda))^{-1} = L(\lambda) + \lambda^{-1} S(\lambda),
$$
with $L(\lambda) \in \W$ and $S(\lambda)^\vee \in \V_{L^{3, 2}, L^{3/2, 2}}$, for sufficiently small $\epsilon>0$.

Furthermore, $0$ is an isolated exceptional point, so $H$ has finitely many negative eigenvalues.

\end{lemma}

\begin{proof}[Proof of Lemma \ref{lemma27}]
We again apply Lemma \ref{lemma_invers} to
$$\begin{aligned}
I + \widehat T(\lambda) &:= 
\bpm \ov Q(I + \widehat T(\lambda)) \ov Q & \ov Q\widehat T(\lambda)Q \\ Q\widehat T(\lambda)\ov Q & Q(I+\widehat T(\lambda))Q \epm \equiv \bpm T_{00}(\lambda) & T_{01}(\lambda) \\ T_{10} (\lambda) & T_{11} (\lambda)\epm.
\end{aligned}$$
The proof of the fact that $\chi(\lambda/\epsilon) T_{00}^{-1}(\lambda) \in \W$ is the same as in Lemma \ref{lemma24}.

Then note that
$$\begin{aligned}
T_{11}(\lambda) &= Q (I+\widehat T(\lambda)) Q = Q (I + V_2 R_0((\lambda+i0)^2) V_1) Q\\
&=Q (V_2 R_0((\lambda+i0)^2) V_1 - V_2 R_0(0) V_1) Q\\
&=V_2 \phi \otimes V\phi(R_0((\lambda+i0)^2) - R_0(0)) V\phi \otimes V_1\phi \\
&:= \lambda c_0(\lambda) Q.
\end{aligned}$$
Observe that $c_0(0) = a^{-1} \ne 0$. Recall that $\ds a=\frac {4i\pi}{|\langle V, \phi\rangle|^2}$.

Note that $c_0(\lambda) \in \widehat L^1$ if
$$
\int_{\R^3} \int_{\R^3} V(x) \phi(x) V(y) \phi(y) \Big\|\frac{e^{i\lambda|x-y|} - 1}{\lambda |x-y|}\Big\|_{\widehat L^1_{\lambda}} \dd x \dd y < \infty.
$$
For every $x$ and $y$, by Lemma \ref{fourier}
$$
\Big\|\frac{e^{i\lambda|x-y|} - 1}{\lambda |x-y|}\Big\|_{\widehat L^1_{\lambda}}=\Big\|\frac {\chi_{[0, |x-y|]}(t)}{|x-y|}\Big\|_{L^1_t}=1,
$$
so it is enough to assume that $V \phi \in L^1$, i.e.\ that $V \in L^{3/2, 1}$, to prove that $c_0(\lambda) \in \widehat L^1$.




Furthermore, recalling that $Q = - V_2 \phi \otimes V_1 \phi$,
\be\lb{t01}\begin{aligned}
T_{01}(\lambda)  &:=\ov Q \widehat T(\lambda) Q = \ov Q (I+\widehat T(\lambda)) Q = (I+\widehat T(\lambda)) Q - Q (I+\widehat T(\lambda)) Q\\
&= -\big(V_2 (R_0((\lambda+i0)^2) -R_0(0)) V \phi + \lambda c_0(\lambda)V_2\phi\big) \otimes V_1 \phi\\
&= -\lambda \Big(V_2 \frac {R_0((\lambda+i0)^2)-R_0(0)}{\lambda} V \phi + c_0(\lambda)V_2\phi\Big) \otimes V_1 \phi.
\end{aligned}\ee

Likewise,
\be\lb{t10}\begin{aligned}
T_{10}(\lambda) &= -V_2\phi \otimes (V_1 (R_0^*((\lambda+i0)^2)-R_0(0)) V \phi +\lambda\ov{c_0(\lambda)}V_1 \phi)\\
&= -\lambda V_2\phi \otimes \Big(V_1 \frac {R_0^*((\lambda+i0)^2)-R_0(0)}{\lambda} V \phi +\ov{c_0(\lambda)}V_1 \phi\Big).
\end{aligned}\ee
Thus $T_{10}^\vee$ and $T_{01}^\vee$ are both in $\V_{L^{3/2, 2}} \cap \V_{L^{3, 2}}$ by the second line of (\ref{t01}), respectively the first line of (\ref{t10}), when $V \in L^{3/2, 1}$. Indeed, following the definition this reduces to
$$
\int_{\R^3} \frac {|V_2(x)| |V(y)| |\phi(y)|}{4\pi|x-y|} \dd y \in L^{3/2, 2}_x \cap L^{3, 2}_x.
$$


Next, $-T_{10} (\lambda) T_{00}^{-1}(\lambda) T_{01} (\lambda) = \lambda c_1(\lambda) Q$, where
\be\begin{aligned}\lb{formula_c1}
c_1(\lambda) 
&=-\Big\langle V_1 (R_0^*((\lambda+i0)^2)-R_0(0)) V\phi+ \lambda \ov{c_0(\lambda)}V_1 \phi, \\
&T_{00}^{-1}(\lambda) \Big(V_2 \frac {R_0((\lambda+i0)^2)-R_0(0)}{\lambda} V\phi+ c_0 (\lambda)V_2\phi\Big) \Big\rangle.
\end{aligned}\ee
For example, one term from formula (\ref{formula_c1}) has the form
\be\lb{termen1}
\Big\langle V_1 (R_0^*((\lambda+i0)^2)-R_0(0)) V_2 V_1\phi, T_{00}^{-1}(\lambda) V_2 \frac {R_0((\lambda+i0)^2)-R_0(0)}{\lambda} V_1 V_2\phi \Big\rangle.
\ee
Note that $V_1 (R_0^*((\lambda+i0)^2)-R_0(0)) V_2$ and $\chi(\lambda/\epsilon) T_{00}^{-1}(\lambda)$ are in $\W$, while  $$
M \Big(V_2 \frac {R_0((\lambda+i0)^2)-R_0(0)}{\lambda} V_1\Big) = \frac{|V_2| \otimes |V_1|}{4\pi} \in \B(L^{3/2, 2}, L^{3, 2}),
$$
so $\ds V_2 \frac {R_0((\lambda+i0)^2)-R_0(0)}{\lambda} V_1 \in \V_{L^{3/2, 2}, L^{3, 2}}$. Taking into account the fact that $V_1 \phi$, $V_2\phi \in L^{3/2, 2}$, it follows that (\ref{termen1}) is in $\widehat L^1$.

Thus we recognize from (\ref{formula_c1}) that $c_1(\lambda) \in \widehat L^1$ when $V \in L^{3/2, 1}$.

Further note that, since when $\lambda=0$ $R_0^*((\lambda+i0)^2) - R_0(0) = 0$, $c_1(0)=0$.
%

Let
$$
C(\lambda) := T_{11}(\lambda) - T_{10}(\lambda)T_{00}^{-1}(\lambda)T_{01} (\lambda).
$$
Then
$$\begin{aligned}
C(\lambda) &= \lambda(c_0(\lambda) + c_1(\lambda)) Q.
\end{aligned}$$
Thus $C(\lambda)/\lambda$ is invertible for $|\lambda|\ll 1$ and 
$C^{-1}(\lambda) = \lambda^{-1} c_2(\lambda) Q$, with $c_2$ locally in $\widehat L^1$. Consequently, for small $\epsilon$ $(\chi(\lambda/\epsilon) \lambda C^{-1}(\lambda))^\vee \in \V_{L^{3, 2}, L^{3/2, 2}}$.

The inverse of $I+\widehat T(\lambda)$ is then given for small $\lambda$ by formula (\ref{invers}):
$$
(I + \widehat T)^{-1}= \bpm T_{00}^{-1} + T_{00}^{-1} T_{01} C^{-1} T_{10} T_{00}^{-1}&-T_{00}^{-1} T_{01} C^{-1} \\ 
-C^{-1} T_{10} T_{00}^{-1} & C^{-1}
\epm.
$$
Since $T_{00}^{-1} \in \W$ and $T_{01}^\vee$, $T_{10}^\vee \in \V_{L^{3/2, 2}} \cap \V_{L^{3, 2}}$, while $(\chi(\lambda/\epsilon) \lambda C^{-1}(\lambda))^\vee \in \V_{L^{3, 2}, L^{3/2, 2}}$, it immediately follows that $\lambda ((I+\widehat T(\lambda))^{-1} - T_{00}^{-1}(\lambda)) \in \V_{L^{3, 2}, L^{3/2, 2}}$ and that $(I+\widehat T)^{-1}$, given by formula (\ref{invers}), exists on a whole cut neighborhood of zero.
\end{proof}

Recall that by (\ref{rs})
$$
R(t) := \frac{ae^{-i\frac{3\pi}4}}{\sqrt{i \pi t}} \zeta_t(x) \otimes \zeta_t(y),\ \zeta_t(x) := e^{i|x|^2/4t} \phi(x).
$$

\begin{proposition}\lb{prop28} Assume that $\langle x \rangle^2 V \in L^{3/2, 1}$ and that $H=-\Delta+V$ is an exceptional Hamiltonian of the first kind with canonical resonance $\phi$ at zero. Then for $1 \leq p < 3/2$ and $R(t)$ as above
$$\begin{aligned}
&e^{-itH} P_c u = Z(t) u + R(t)u,\\
&\|Z(t)u\|_{L^{p'}} \les t^{-\frac 3 2(\frac 1 p - \frac 1 {p'})} \|f\|_{L^p},\ \|Z(t)u\|_{L^{3, \infty}} \les t^{-\frac 1 2} \|f\|_{L^{3/2, 1}}.
\end{aligned}$$
Furthermore, for $3/2<p\leq 2$
$$
\|e^{-itH} P_c u\|_{L^{p'}} \les t^{-\frac 3 2(\frac 1 p - \frac 1 {p'})} \|u\|_{L^p}.
$$
Here $\frac 1 p + \frac 1 {p'} = 1$.
\end{proposition}
\begin{proof}[Proof of Proposition \ref{prop28}]
Write the evolution as
$$
e^{-itH}P_c f = \frac 1 {i\pi} \int_\R e^{-it\lambda^2} \big(R_0((\lambda+i0)^2) - R_0((\lambda+i0)^2) V_1 \widehat T(\lambda)^{-1}  V_2 R_0((\lambda+i0)^2)\big) f \lambda \dd \lambda.
$$
We consider a partition of unity subordinated to the neighborhoods of Lemmas \ref{lemma23} and \ref{lemma24}. First, take a sufficiently large $R$ such that $(1-\chi(\lambda/R)) (I+\widehat T(\lambda))^{-1} \in \W$. Then, for every $\lambda_0 \in [-4R, 4R]$ there exists $\epsilon(\lambda_0)>0$ such that $\chi(\frac{\lambda-\lambda_0}{\epsilon(\lambda_0)}(I+\widehat T(\lambda))^{-1} \in \W$ if $\lambda_0\ne 0$ or the conclusion of Lemma \ref{lemma24} holds when $\lambda_0=0$.

Since $[-4R, 4R]$ is a compact set, there exists a finite covering $[-4R, 4R] \subset \bigcup_{k=1}^N (\lambda_k - \epsilon(\lambda_k), \lambda_k + \epsilon(\lambda_k))$. Then we construct a finite partition of unity on $\R$ by smooth functions $1= \chi_0(\lambda) + \sum_{k=1}^N \chi_k(\lambda) + \chi_\infty(\lambda)$, where $\supp \chi_\infty \subset \R \setminus (-2R, 2R)$, $\supp \chi_0 \subset [-\epsilon(0), \epsilon(0)]$, and $\supp \chi_k \subset [\lambda_k-\epsilon(\lambda_k), \lambda_k+\epsilon(\lambda_k)]$.

By Lemma \ref{lemma23}, for any $k \ne 0$, $\chi_k(\lambda) (I+\widehat T (\lambda))^{-1} \in \W$, so $(1-\chi_0(\lambda)) (I+\widehat T(\lambda))^{-1} \in \W$. By Lemma \ref{lemma24} $\chi_0(\lambda) \widehat T(\lambda)$ also decomposes into a regular term $L \in \W$ and a singular term $-\lambda^{-1} \chi_0(\lambda) a V_2 \phi \otimes V_1 \phi$.

Let $Z_1$ be given by the sum of all the regular terms in the decomposition:
$$\begin{aligned}
Z_1(t) &:= \frac 1 {i\pi} \int_\R e^{-it\lambda^2} \big(R_0((\lambda+i0)^2) - R_0((\lambda+i0)^2) V_1 L(\lambda) V_2 R_0((\lambda+i0)^2) - \\
&- (1-\chi_0(\lambda)) R_0((\lambda+i0)^2) V_1 \widehat T(\lambda) V_2 R_0((\lambda+i0)^2) \big) \lambda \dd \lambda \\
&= \frac 1 {2\pi t} \int_\R e^{-it\lambda^2} \partial_\lambda \big(R_0((\lambda+i0)^2) - R_0((\lambda+i0)^2) V_1 L(\lambda) V_2 R_0((\lambda+i0)^2) -\\
&-(1-\chi_0(\lambda)) R_0((\lambda+i0)^2) V_1 \widehat T(\lambda) V_2 R_0((\lambda+i0)^2) \big) \dd \lambda \\
&= \frac C {t^{3/2}} \int_\R e^{-i\frac{\rho^2}{4t}} \big(\partial_\lambda \big(R_0((\lambda+i0)^2) - R_0((\lambda+i0)^2) V_1 L(\lambda) V_2 R_0((\lambda+i0)^2) -\\
&-(1-\chi_0(\lambda)) R_0((\lambda+i0)^2) V_1 \widehat T(\lambda) V_2 R_0((\lambda+i0)^2) \big)\big)^\vee(\rho) \dd \rho.
\end{aligned}$$
The fact that $\|Z_1(t) u\|_{L^\infty} \les |t|^{-3/2} \|u\|_{L^1}$ follows by knowing that
$$\begin{aligned}
\big(\partial_\lambda \big(&R_0((\lambda+i0)^2) - R_0((\lambda+i0)^2) V_1 L(\lambda) V_2 R_0((\lambda+i0)^2) - \\
&-(1-\chi_0(\lambda)) R_0((\lambda+i0)^2) V_1 \widehat T(\lambda) V_2 R_0((\lambda+i0)^2)\big)\big)^\vee \in \V_{L^1, L^{\infty}}.
\end{aligned}$$
The fact that $\|Z_1(t) u\|_{L^2} \les \|u\|_{L^2}$ follows by smoothing estimates. Indeed, the first term is bounded since it represents the free evolution and note that
$$\begin{aligned}
&\|V_2 R_0(\lambda\pm i0) f\|_{L^2_{\lambda, x}} \les \|f\|_{L^2_x},\\
&\|e^{-it\lambda} (L(\pm\sqrt\lambda)+(1-\chi_0(\pm\sqrt\lambda)) \widehat T(\pm\sqrt\lambda))\|_{L^\infty_\lambda \B(L^2)} < \infty,\\
&\Big\|\int_{\R} R_0(\lambda \pm i0) V_1 F(x, \lambda) \dd \lambda\Big\|_{L^2_x} \les \|F\|_{L^2_{\lambda, x}}.
\end{aligned}$$
Combining these three estimates we obtain the $L^2$ boundedness of $Z_1$.

By interpolation between the two bounds we obtain that, for $\frac 1 p + \frac 1 {p'} = 1$, $1 \leq p \leq 2$,
$$
\|Z_1(t) u\|_{L^{p'}} \les t^{-\frac 3 2 (\frac 1 p -\frac 1 {p'})} \|u\|_{L^p},
$$
as well as
$$
\|Z_1(t) u\|_{L^{3, \infty}} \les t^{-1/2} \|u\|_{L^{3/2, 1}}.
$$

Let $Z_2$ be the term corresponding to the singular part of the decomposition from Lemma \ref{lemma24}, given~by
$$\begin{aligned}
Z_2(t) &:= \frac {a}{i\pi} \int_{\R} e^{-it\lambda^2}\chi_0(\lambda) R_0((\lambda+i0)^2) V \phi \otimes V \phi R_0((\lambda+i0)^2) \dd \lambda \\
&= \frac{a}{i\pi} \int_\R \int_{(\R^3)^2} e^{-it\lambda^2} \chi_0(\lambda) \frac {e^{i\lambda|x-z_1|}}{4\pi|x-z_1|}V(z_1) \phi(z_1) V(z_2) \phi(z_2) \frac {e^{i\lambda |z_2-y|}}{4\pi|z_2-y|} \dd z_1 \dd z_2 \dd \lambda.
\end{aligned}$$

The subsequent Lemma \ref{lemma4.10} is the same as Lemma 4.10 from \cite{yajima_disp}, the only difference being the space of potentials for which the result holds. For the sake of completeness we repeat the proof given in \cite{yajima_disp}.
\begin{lemma}\lb{lemma4.10} For $V \in \langle x \rangle^{-1} L^{3/2, 1}$
\be\lb{z2}
\|(Z_2(t) - R(t)) u\|_{L^{\infty}} \les t^{-3/2} \|u\|_{L^1},\ \|Z_2(t) u\|_{L^{3, \infty}} \les t^{-1/2} \|u\|_{L^{3/2, 1}}.
\ee
\end{lemma}

\begin{proof}[Proof of Lemma \ref{lemma4.10}]
Let $b = |x-z_1| + |z_2-y|$ and
$$\begin{aligned}
C(t, b) &= \frac 1 {i\pi} \int_\R e^{-it\lambda^2+i\lambda b} \chi_0(\lambda) \dd \lambda.
\end{aligned}$$
We express $Z_2(t)$ as
$$
Z_2(t) = \int_{(\R^3)^2} C(t, b) a \frac {V(z_1) \phi(z_1) V(z_2) \phi(z_2)} {|x-z_1| |z_2-y|} \dd z_1 \dd z_2.
$$
Note that
$$\begin{aligned}
C(t, b)&=\frac {e^{-i\frac{3\pi}4} e^{i\frac{b^2}{4t}}}{\sqrt{\pi t}} \big(e^{i\frac{s^2}{4t}} \chi_0^{\vee}(s)\big)^{\wedge}(b/2t).
\end{aligned}$$
Then $C(t, b) \les t^{-1/2}$ and
$$
|Z_2(t)(x, y)| \les t^{-1/2} \int_{(\R^3)^2} \frac {|V(z_1) \phi(z_1) V(z_2) \phi(z_2)|}{|z_1-x| |z_2-y|} \dd z_1 \dd z_2.
$$
Clearly $\ds\int_{\R^3} \frac {|V(z_1) \phi(z_1)|}{|z_1-x|} \dd z_1 \in L^{3, \infty}_x$ and $\ds\int_{\R^3} \frac {|V(z_2) \phi(z_2)|}{|z_2-y|} \dd z_2 \in L^{3, \infty}_y$, implying the second half of (\ref{z2}):
$$
\|Z_2(t) u\|_{L^{3, \infty}} \les t^{-1/2} \|u\|_{L^{3/2, 1}}.
$$
We also have
$$
\Big| \mc F(e^{\frac{is^2}{4t}}\chi_0^{\vee}(s))(b/2t) - 1\Big| \les t^{-1} (\|s^2\chi_0^{\vee}\|_1 + |b|).
$$
It is easy to see, for
$$
B = 2(|x - z_1 ||z_1 | + |z_2 - y ||z_2 | + |x-z_1| |z_2-y|) + |z_1 |^2 + |z_2 |^2,
$$
that
$$
|e^{ib^2/4t}-e^{i(x^2+y^2)/4t}| = |e^{i(|x-z_1| + |z_2-y|)^2/4t}-e^{i(x^2+y^2)/4t}| \leq \frac{B}{4t}.
$$
It follows that
$$
C(t, b) - \frac{e^{-i\frac{3\pi}4}e^{i(x^2+y^2)/4t}}{\sqrt{\pi t}} \les (1+b+B) t^{-3/2}.
$$
Then
$$\begin{aligned}
&\Big|Z_2(t) - \int_{(\R^3)^2} \frac{e^{i\frac {3\pi}4} e^{i(x^2+y^2)/4t}}{\sqrt{\pi t}} a \frac{V(z_1)\phi(z_1)V(z_2)\phi(z_2) \dd z_1 \dd z_2}{|x-z_1| |y-z_2|} \Big|\les  \\
&\les t^{-3/2} \int_{(\R^3)^2} \frac{(1+b+B)|V(z_1) \phi(z_1)V(z_2)\phi(z_2)|}{|x-z_1| |z_2-y|} \dd z_1 \dd z_2.
\end{aligned}$$
Now note that for $V \in \langle x \rangle^{-1} L^{3/2, 1}$ and $\phi(x) \les |x|^{-1}$
$$
\sup_{x, y} \int_{(\R^3)^2} \frac{(1+b+B)|V(z_1) \phi(z_1)V(z_2)\phi(z_2)|}{|x-z_1| |z_2-y|} \dd z_1 \dd z_2 < \infty
$$
and
$$
\int_{\R^3} \frac{V(z_1) \phi(z_1) \dd z_1}{|x-z_1|} = \phi(x).
$$
The first part of conclusion (\ref{z2}) follows.
\end{proof}

Note that $R(t)$ also satisfies $\|R(t) u\|_{L^{3, \infty}} \les t^{-1/2} \|u\|_{L^{3/2, 1}}$, so the same holds for the difference:
$$
\|(Z_2(t)-R(t)) u\|_{L^{3, \infty}} \les t^{-1/2} \|u\|_{L^{3/2, 1}}
$$
By interpolation with the $L^1$ to $L^{\infty}$ estimate of Lemma \ref{lemma4.10} we obtain that for $1 \leq p < 3/2$
$$
\|(Z_2(t)-R(t)) u\|_{L^{p'}} \les t^{-\frac 3 2(\frac 1 p - \frac 1 {p'})} \|u\|_{L^p}.
$$
Since the same is true for $Z_1$, we obtain for $1 \leq p < 3/2$ that
$$
\|(Z(t) u\|_{L^{p'}} = \|(Z_1(t)+Z_2(t)-R(t)) u\|_{L^{p'}} \les t^{-\frac 3 2(\frac 1 p - \frac 1 {p'})} \|u\|_{L^p},
$$
where $e^{-itH} P_c u = Z_1(t) u + Z_2(t) u = Z(t) u + R(t) u$.

Knowing that $\|Z_i(t) u\|_{L^{3, \infty}} \les t^{-1/2} \|u\|_{L^{3/2, 1}}$ leads to the conclusion that $\|e^{-itH} P_c u\|_{L^{3, \infty}} \les \|u\|_{L^{3/2, 1}}$. Combining this with the $L^2$ estimate $\|e^{-itH} P_c u\|_{L^2} \les \|u\|_{L^2}$, we obtain that for $3/2<p \leq 2$
$$
\|e^{-itH} P_c u\|_{L^{p'}} \les t^{-\frac 3 2(\frac 1 p - \frac 1 {p'})} \|u\|_{L^p}.
$$
Thus we have proved all the conclusions of Proposition \ref{prop28}.
\end{proof}

\begin{proposition}\lb{prop29} Assume that $V \in L^{3/2, 1}$ and that $H=-\Delta+V$ is an exceptional Hamiltonian of the first kind. Then
$$
\|e^{-itH}P_c u\|_{L^{3, \infty}} \les t^{-1/2} \|u\|_{L^{3/2, 1}}
$$
and for $3/2<p\leq 2$
$$
\|e^{-itH} P_c u\|_{L^{p'}} \les t^{-\frac 3 2(\frac 1 p - \frac 1 {p'})} \|u\|_{L^p}.
$$
Here $\frac 1 p + \frac 1 {p'} = 1$.
\end{proposition}
\begin{proof}[Proof of Proposition \ref{prop29}]
Write the evolution as
$$
e^{-itH}P_c f = \frac 1 {i\pi} \int_\R e^{-it\lambda^2} \big(R_0((\lambda+i0)^2) - R_0((\lambda+i0)^2) V_1 \widehat T(\lambda)^{-1}  V_2 R_0((\lambda+i0)^2)\big) f \lambda \dd \lambda.
$$
We consider a partition of unity subordinated to the neighborhoods of Lemmas \ref{lemma23} and \ref{lemma27}. First, take a sufficiently large $R$ such that $(1-\chi(\lambda/R)) (I+\widehat T(\lambda))^{-1} \in \W$. Then, for every $\lambda_0 \in [-4R, 4R]$ there exists $\epsilon(\lambda_0)>0$ such that $\chi(\frac{\lambda-\lambda_0}{\epsilon(\lambda_0)}(I+\widehat T(\lambda))^{-1} \in \W$ if $\lambda_0\ne 0$ or the conclusion of Lemma \ref{lemma27} holds when $\lambda_0=0$.

Since $[-4R, 4R]$ is a compact set, there exists a finite covering $[-4R, 4R] \subset \bigcup_{k=1}^N (\lambda_k - \epsilon(\lambda_k), \lambda_k + \epsilon(\lambda_k))$. Then we construct a finite partition of unity on $\R$ by smooth functions $1= \chi_0(\lambda) + \sum_{k=1}^N \chi_k(\lambda) + \chi_\infty(\lambda)$, where $\supp \chi_\infty \subset \R \setminus (-2R, 2R)$, $\supp \chi_0 \subset [-\epsilon(0), \epsilon(0)]$, and $\supp \chi_k \subset [\lambda_k-\epsilon(\lambda_k), \lambda_k+\epsilon(\lambda_k)]$.

By Lemma \ref{lemma23}, for any $k \ne 0$, $\chi_k(\lambda) (I+\widehat T (\lambda))^{-1} \in \W$, so $(1-\chi_0(\lambda)) (I+\widehat T(\lambda))^{-1} \in \W$. By Lemma \ref{lemma27} $\chi_0(\lambda) (I+\widehat T(\lambda))^{-1}$ also decomposes into a regular term $L \in \W$ and a singular term $\lambda^{-1} S$ with the property that $S^\vee \in \V_{L^{3, 2}, L^{3/2, 2}}$.


Let $Z_1$ be given by the sum of all the regular terms of the decomposition:
$$\begin{aligned}
Z_1(t) &:= \frac 1 {i\pi} \int_\R e^{-it\lambda^2} \big(R_0((\lambda+i0)^2) - R_0((\lambda+i0)^2) V_1 L(\lambda) V_2 R_0((\lambda+i0)^2) - \\
&- (1-\chi_0(\lambda)) R_0((\lambda+i0)^2) V_1 \widehat T(\lambda) V_2 R_0((\lambda+i0)^2) \big) \lambda \dd \lambda \\
&= \frac 1 {2\pi t} \int_\R e^{-it\lambda^2} \partial_\lambda \big(R_0((\lambda+i0)^2) - R_0((\lambda+i0)^2) V_1 L(\lambda) V_2 R_0((\lambda+i0)^2) - \\
&- (1-\chi_0(\lambda)) R_0((\lambda+i0)^2) V_1 \widehat T(\lambda) V_2 R_0((\lambda+i0)^2) \big) \dd \lambda\\
&= \frac C {t^{3/2}} \int_\R e^{-i\frac{\rho^2}{4t}} \big(\partial_\lambda \big(R_0((\lambda+i0)^2) - R_0((\lambda+i0)^2) V_1 L(\lambda) V_2 R_0((\lambda+i0)^2) - \\
&- (1-\chi_0(\lambda)) R_0((\lambda+i0)^2) V_1 \widehat T(\lambda) V_2 R_0((\lambda+i0)^2) \big)\big)^\vee(\rho) \dd \rho.
\end{aligned}$$
The fact that $\|Z_1(t) u\|_{L^\infty} \les |t|^{-3/2} \|u\|_{L^1}$ follows by knowing that
$$\begin{aligned}
\big(\partial_\lambda \big(&R_0((\lambda+i0)^2) - R_0((\lambda+i0)^2) V_1 L(\lambda) V_2 R_0((\lambda+i0)^2) - \\
&- (1-\chi_0(\lambda)) R_0((\lambda+i0)^2) V_1 \widehat T(\lambda) V_2 R_0((\lambda+i0)^2)\big)\big)^\vee \in \V_{L^1, L^{\infty}}.
\end{aligned}$$

By smoothing estimates it immediately follows that $Z_1(t)$ is $L^2$-bounded, see the proof of Proposition \ref{prop28}. Interpolating we obtain the desired $\|Z_1(t)u\|_{L^{3, \infty}} \les t^{-1/2} \|u\|_{L^{3, 1}}$ estimate.

Let $Z_2$ be the singular part of the decomposition from Lemma \ref{lemma27}, given~by
\be\begin{aligned}\lb{zzz}
Z_2(t) &:= \frac{1}{i\pi} \int_{\R} e^{-it\lambda^2} R_0((\lambda+i0)^2) V_1 S(\lambda) V_2 R_0((\lambda+i0)^2) \dd \lambda.
\end{aligned}\ee
Note that $(R_0((\lambda+i0)^2) V_1)^\vee \in \V_{L^{3/2, 2}, L^{3, \infty}}$, $S(\lambda)^\vee \in \V_{L^{3, 2}, L^{3/2, 2}}$, and $(V_2 R_0((\lambda+i0)^2))^\vee \in \V_{L^{3/2, 1}, L^{3, 2}}$. Thus
$$
R_0((\lambda+i0)^2) V_1 (\lambda S(\lambda)) V_2 R_0((\lambda+i0)^2) \in \V_{L^{3/2, 1}, L^{3, \infty}}.
$$
By taking the Fourier transform in (\ref{zzz}), this immediately implies the conclusion that $\|Z_2(t) u\|_{L^{3, \infty}} \les t^{-1/2} \|u\|_{L^{3/2, 1}}$.

Putting the two estimates for $Z_1$ and $Z_2$ together, we obtain that $\|e^{-itH}P_c u\|_{L^{3, \infty}} \les \|u\|_{L^{3/2, 1}}$. Interpolating with the obvious $L^2$ bound $\|e^{-itH} P_c u\|_{L^2} \les \|u\|_{L^2}$, we obtain the stated conclusion.
\end{proof}

\subsection{Exceptional Hamiltonians of the third kind} We next consider the case in which $H$ is exceptional of the third kind, i.e.\ there are both zero eigenvectors and zero resonances. Recall that $\widehat T(\lambda) = V_2 R_0((\lambda+i0)^2) V_1$.
\begin{lemma}\lb{lemma28} Suppose that $V \in \langle x \rangle^{-4} L^{3/2, 1}$ and $H=-\Delta+V$ has both eigenvectors and resonances at zero. Let $\chi$ be a standard cutoff function. Then for sufficiently small $\epsilon$
$$\begin{aligned}
\chi(\lambda/\epsilon) (I + \widehat T(\lambda))^{-1} &= L(\lambda) + \chi(\lambda/\epsilon)\Big(\frac{V_2 P_0 V_1}{\lambda^2} + \frac{i V_2 P_0 V |x-y|^2 V P_0 V_1}{\lambda} - \\
&- \frac {a V_2\phi \otimes V_1\phi}{\lambda}\Big),
\end{aligned}$$
where $L(\lambda) \in \W$ and $\phi$ is a certain resonance for $H=-\Delta+V$.

Furthermore, $0$ is an isolated exceptional point for $H$, meaning that $H$ has finitely many negative eigenvalues.
\end{lemma}
The computations in the proof of this lemma parallel those in Section 4.5 of \cite{yajima_disp}. The main difference is using the space $\W$ instead of H\"{o}lder spaces.
\begin{proof}[Proof of Lemma \ref{lemma28}]
We study $(I+\widehat T(\lambda))^{-1} := (I + V_2R_0((\lambda+i0)^2)V_1)^{-1}$ near $\lambda = 0$.

Let
$$
Q = - \frac 1 {2\pi i} \int_{|z+1|=\delta} (V_2 R_0(0) V_1 - zI)^{-1} \dd z
$$

Take the orthonormal basis $\{\phi_1 , . . . , \phi_N\}$ with respect to the inner product $-(V u, v)$ for $\mc M$ so that $\{\phi_2 , \ldots , \phi_N\}$ is a basis of $\mc E$ and $\langle \phi_1 , V\rangle > 0$. This condition determines $\phi_1$ uniquely.

Define the orthogonal projections $\pi_1$ onto $\C V_1 \phi_1$ and $\pi_2$ onto $V_1 P_0 L^2$ with respect to the inner product $-(\sgn V u, v)$, i.e.\ $\pi_1 = -V_2 \phi_1\otimes V_1 \phi_1$ and $\pi_2 = - \sum_{j=2}^N V_2\phi_j \otimes V_1 \phi_j$ and let
$$
Q_0 = \ov Q := 1 - Q,\ Q_1 := Q \pi_1 Q,\ Q_2 := Q \pi_2 Q.
$$



The following identities hold in $L^2$:
$$
\begin{aligned}
&Q_j Q_k = \delta_{j k} I,\ j, k = 0, 1, 2 \text{ and } Q_0 + Q_1 + Q_2 = I,\\
&(I + V_2 R_0(0) V_1)Q_1 = Q_1(I+V_2 R_0(0) V_1) = 0,\\
&(I + V_2 R_0(0) V_1)Q_2 = Q_2(I+V_2 R_0(0) V_1) = 0,\\
&Q_2 (V_2 \otimes V_1) Q_0 = 0,\ Q_2 (V_2 \otimes V_1) Q_1 = 0,\ Q_2 (V_2 \otimes V_1) Q_2 = 0,\\ 
&Q_0 (V_2 \otimes V_1) Q_2 =0,\ Q_1 (V_2\otimes V_1) Q_2=0.
\end{aligned}
$$
These identities follow from $Q_2 V_2 = 0$ and $Q_2^* V_1 = 0$, which in turn follow from the fact that eigenvectors $\phi_k$ are orthogonal to $V$, $\langle \phi_k, V \rangle=0$, for $2 \leq k \leq N$.

We first apply Lemma \ref{lemma_invers} to invert $Q(I+\widehat T(\lambda))Q$ in $Q L^2$ for small $\lambda$, after writing it in matrix form with respect to the decomposition $QL^2 = Q_1 L^2 + Q_2 L^2$:
$$
Q(I+\widehat T(\lambda))Q = \bpm 
Q_1 (I+\widehat T(\lambda)) Q_1 &Q_1 \widehat T(\lambda) Q_2\\
Q_2 \widehat T(\lambda) Q_1 &Q_2 (I+\widehat T(\lambda)) Q_2
\epm := \bpm
T_{11}(\lambda) &T_{12}(\lambda)\\
T_{21}(\lambda) &T_{22}(\lambda)\epm.
$$
The inverse will be given by formula (\ref{invers}), that is
\be\lb{invers'}
(Q(I+\widehat T(\lambda))Q)^{-1} = \bpm T_{11}^{-1} + T_{11}^{-1} T_{12} C_{22}^{-1} T_{21} T_{11}^{-1}&-T_{11}^{-1} T_{12} C_{22}^{-1} \\ 
-C_{22}^{-1} T_{21} T_{11}^{-1} & C_{22}^{-1}
\epm,
\ee
where
$$
C_{22} = T_{22} - T_{21} T_{11}^{-1} T_{12}.
$$
As in the case of exceptional Hamiltonians of the first kind, let
$$
J(\lambda):=\frac{\widehat T(\lambda)- (V_2 R_0(0) V_1 + i \lambda(4\pi)^{-1}V_2 \otimes V_1)}{\lambda^2}.
$$
Then (recall that $Q_1=-V_2\phi_1 \otimes V_1 \phi_1$)
$$\begin{aligned}
T_{11}(\lambda) &= Q_1 (I+\widehat T(\lambda)) Q_1 = Q_1 (I + V_2 R_0((\lambda+i0)^2) V_1) Q_1\\
&=Q_1 (V_2 R_0((\lambda+i0)^2) V_1 - V_2 R_0(0) V_1) Q_1\\
&=V_2 \phi_1 \otimes V\phi_1 (R_0((\lambda+i0)^2) - R_0(0)) V \phi_1 \otimes V_1 \phi_1 \\
&=\Big(\lambda \frac {|\langle V, \phi_1 \rangle|^2}{4i\pi} - \lambda^2 \langle V_1 \phi_1, J(\lambda) V_2 \phi_1 \rangle\Big) Q_1 \\
&:= (\lambda a^{-1} + \lambda^2 c_1(\lambda)) Q_1.
\end{aligned}$$
Here $\ds a=\frac {4i\pi}{|\langle V, \phi_1\rangle|^2} \ne 0$. As in the proof of Lemma \ref{lemma24}, note that $c_1(\lambda) \in \widehat L^1$ when $V \in L^1$ and  $\partial_\lambda c_1(\lambda) \in \widehat L^1$ when $V \in \langle x \rangle^{-1} L^1$.

It follows that $T_{11}(\lambda)$ is invertible for $|\lambda| \ll 1$ in $Q_1 L^2$ and
$$
T_{11}^{-1}(\lambda) = \frac 1 {\lambda a^{-1} + \lambda^2 c_1(\lambda)} Q_1 = \Big(\frac a \lambda - \frac {c_1(\lambda)}{(a^{-1}+\lambda c_1(\lambda))a^{-1}}\Big) Q_1 = \lambda^{-1} a Q_1 + E(\lambda).
$$
Here and below we denote by $E(\lambda)$ various regular terms, i.e.\ with the property that $\chi(\lambda/\epsilon) E(\lambda) \in \W$ for sufficiently small $\epsilon$.

Likewise, since $Q_2 (V_2 \otimes V_1) = (V_2 \otimes V_1) Q_2 = 0$,
$$\begin{aligned}
T_{12}(\lambda) &= Q_1 (I+V_2 R_0((\lambda+i0)^2) V_1) Q_2\\
&=Q_1 V_2\Big(R_0((\lambda+i0)^2) - R_0(0) - i \lambda \frac{1 \otimes 1}{4\pi}\Big) V_1 Q_2\\
&=-\lambda^2 Q_1 \Big(V_2\frac{|x-y|}{8\pi} V_1 + \lambda V_2e_1(\lambda)V_1\Big)Q_2 \\
&=-\lambda^2 Q_1 V_2\frac{|x-y|}{8\pi} V_1 Q_2 + \lambda^3 E(\lambda),
\end{aligned}$$
where
$$
e_1(\lambda) := \frac{\ds R_0((\lambda+i0)^2) - R_0(0) - i \lambda \frac{1 \otimes 1}{4\pi} + \lambda^2 \frac{|x-y|}{8\pi}}{-\lambda^3}.
$$
By Lemma \ref{fourier} $\ds M((e_1(\lambda))^\wedge) = \frac{|x-y|^2}{24\pi}$ and $\ds M((\partial_\lambda e_1(\lambda))^\wedge) = \frac{|x-y|^3}{96\pi}$. Thus $E(\lambda) := Q_1 V_2e_1(\lambda) V_1Q_2 \in \W$ when
$$
\int_{(\R^3)^2} V(x) \phi_1(x) |x-y|^3 V(y) \phi_k(y) \dd x \dd y < \infty,
$$
which takes place when $V \in \langle x \rangle^{-2} L^1$ (recall that $|\phi_1(y)| \les \langle y \rangle^{-1}$).


Likewise we obtain
$$\begin{aligned}
T_{21}(\lambda) &= -\lambda^2 Q_2 V_2 \frac{|x-y|}{8\pi} V_1 Q_1 + \lambda^3 E(\lambda),
\end{aligned}$$
hence, combining the previous results,
$$\begin{aligned} 
T_{21}(\lambda) T_{11}^{-1}(\lambda) T_{12}(\lambda) &= \lambda^3 a Q_2 V_2\frac{|x-y|}{8\pi} V_1 Q_1 V_2 \frac{|x-y|}{8\pi} V_1 Q_2 + E(\lambda).
\end{aligned}$$
Furthermore
$$\begin{aligned}
T_{22}(\lambda) &= Q_2 (I+V_2 R_0((\lambda+i0)^2) V_1) Q_2 \\
&=Q_2 V_2\Big(R_0((\lambda+i0)^2) - R_0(0) - i \lambda \frac{1 \otimes 1}{4\pi}\Big) V_1 Q_2\\
&=-\lambda^2 Q_2 \Big(V_2\frac{|x-y|}{8\pi} V_1 + i \lambda V_2\frac {|x-y|^2}{24\pi} V_1 - \lambda^2 V_2 e_2(\lambda) V_1 \Big)Q_2. 
\end{aligned}$$
Here
$$
e_2(\lambda) := \lambda^{-4}\Big(R_0((\lambda+i0)^2) - R_0(0) - i \lambda \frac{1 \otimes 1}{4\pi} - i \lambda^2 \frac{|x-y|}{8\pi} - \lambda^3\frac{|x-y|^2}{24\pi}\Big).
$$
By Lemma \ref{fourier} $\ds M(e_2(\lambda)^\wedge) = \frac{|x-y|^3}{96\pi}$ and $\ds M((\partial_\lambda e_2(\lambda))^\wedge) = \frac{|x-y|^4}{480\pi}$. Thus $E(\lambda):=Q_2 V_2 e_2(\lambda) V_1 Q_2 \in \W$ when
$$
\int_{(\R^3)^2} V(x) \phi_k(x) |x-y|^4 V(y) \phi_k(y) \dd x \dd y < \infty,
$$
which holds true when $V \in \langle x \rangle^{-2} L^1$ (recall that $|\phi_k(y)| \les \langle y \rangle^{-2}$). Then
\be\lb{m22}
T_{22}(\lambda) = -\lambda^2 Q_2 \Big(V_2\frac{|x-y|}{8\pi} V_1 + i \lambda V_2\frac {|x-y|^2}{24\pi} V_1\Big)Q_2 + \lambda^4 E(\lambda).
\ee

Let $P_0$ be the $L^2$ orthogonal projection onto the set $\mc E$ spanned by $\phi_2, \ldots, \phi_N$. By relation (4.38) of \cite{yajima_disp},
$$
\Big(Q_2 V_2\frac{|x-y|}{8\pi} V_1 Q_2 \Big)^{-1} = -V_2P_0 V_1.
$$
Also note that
$$
V_2 P_0 V_1 Q_2 = Q_2 V_2 P_0 V_1 = V_2 P_0 V_1.
$$
By (\ref{m22}),
$$\begin{aligned}
T_{22}^{-1}(\lambda) &= -\lambda^{-2} \Big( Q_2 V_2 \frac {|x-y|}{8\pi} V_1 Q_2 \Big)^{-1} \sum_{k=0}^\infty (-1)^k \Big(\Big(i\lambda Q_2 V_2 \frac {|x-y|^2}{24\pi} V_1 Q_2 - \lambda^2 E(\lambda)\Big) \\
&\Big( Q_2 V_2 \frac {|x-y|}{8\pi} V_1 Q_2 \Big)^{-1}\Big)^k \\
&=\lambda^{-2} V_2 P_0 V_1 \sum_{k=0}^\infty \Big(\Big(i\lambda V_2 \frac {|x-y|^2}{24\pi} V_1 - \lambda^2 E(\lambda)\Big) V_2 P_0 V_1.
\end{aligned}$$
Therefore, by grouping the terms by the powers of $\lambda$, for $|\lambda| \ll 1$
$$
T_{22}^{-1}(\lambda)=\lambda^{-2} V_2P_0 V_1 + i \lambda^{-1} V_2P_0 V \frac{|x-y|^2}{24\pi} V P_0 V_1 + E(\lambda). 
$$
Then we write
$$\begin{aligned}
C_{22}(\lambda) &= T_{22}(\lambda) - T_{21}(\lambda) T^{-1}_{11}(\lambda) T_{12}(\lambda) \\
&= (I - T_{21}(\lambda) T^{-1}_{11}(\lambda) T_{12}(\lambda)T_{22}^{-1}(\lambda)) T_{22}(\lambda).
\end{aligned}$$
By our previous estimates, $T_{21}(\lambda) T^{-1}_{11}(\lambda) T_{12}(\lambda) T_{22}^{-1}(\lambda) = \lambda E(\lambda)$, $E(\lambda) \in \W$. Then by means of a Neumann series expansion we retrieve that
$$\begin{aligned}
C_{22}^{-1}(\lambda) &= T_{22}^{-1}(\lambda) \sum_{k=0}^\infty (T_{21}(\lambda) T^{-1}_{11}(\lambda) T_{12}(\lambda) T_{22}^{-1}(\lambda))^k \\
&= T_{22}^{-1}(\lambda) + T_{22}^{-1}(\lambda) T_{21}(\lambda) T^{-1}_{11}(\lambda) T_{12}(\lambda) T_{22}^{-1}(\lambda) + E(\lambda),
\end{aligned}$$
so
$$\begin{aligned}
C_{22}^{-1}(\lambda) &= \lambda^{-2}V_2P_0V_1 + i\lambda^{-1} V_2P_0 V \frac{|x-y|^2}{24\pi} V P_0 V_1 + \\
&+a \lambda^{-1} V_2P_0 V \frac{|x-y|}{8\pi} V_1 Q_1 V_2 \frac{|x-y|}{8\pi} V P_0 V_1 + E(\lambda).
\end{aligned}$$
If we set $\ds \tilde \phi_1 = P_0 V \frac {|x-y|}{8\pi} V \phi_1 \in \mc E$, then
$$
V_2 P_0 V \frac{|x-y|}{8\pi} V_1 Q_1 V_2 \frac{|x-y|}{8\pi} V P_0 V_1 = -V_2 \tilde \phi_1 \otimes \tilde \phi_1 V_1.
$$
Then we get that
$$\begin{aligned}
C_{22}^{-1}(\lambda) &= \lambda^{-2}V_2P_0V_1 + i\lambda^{-1} V_2 P_0 V \frac{|x-y|^2}{24\pi} V P_0 V_1 - \lambda^{-1} aV_2 \tilde \phi_1 \otimes \tilde \phi_1V_1 + E(\lambda).
\end{aligned}$$
Furthermore
$$\begin{aligned}
-T_{11}^{-1}(\lambda) T_{12}(\lambda) C_{22}^{-1}(\lambda) &= (\lambda^{-1} a Q_1 + E(\lambda)) \lambda^2 Q_1\Big(V_2\frac{|x-y|}{8\pi}V_1 + \lambda E(\lambda)\Big)Q_2 \\
&\big(\lambda^{-2}V_2P_0V_1 + i\lambda^{-1} E(\lambda)\big)\\
&= \lambda^{-1} a (-V_2 \phi_1 \otimes V_1 \phi_1) V_2\frac{|x-y|}{8\pi} V P_0V_1 + E(\lambda)\\
&= -a\lambda^{-1} V_2\phi_1 \otimes \tilde \phi_1 V_1 + E(\lambda).
\end{aligned}$$
Likewise we obtain
$$\begin{aligned}
-C_{22}^{-1}(\lambda) T_{21}(\lambda) T_{11}^{-1}(\lambda) &= -a\lambda^{-1} V_2 \tilde \phi_1 \otimes \phi_1 V_1 + E(\lambda)\\
T_{11}^{-1}(\lambda) T_{12}(\lambda) C_{22}^{-1}(\lambda) T_{21}(\lambda) T^{-1}_{11}(\lambda) &= E(\lambda).
\end{aligned}$$
By (\ref{invers'}), $(Q(I+\widehat T(\lambda))Q)^{-1}$ is given in matrix form modulo $E(\lambda) \in \W$ by
\be\lb{QQ}
\bpm
-a\lambda^{-1} V_2\phi_1 \otimes V_1 \phi_1 & -a\lambda^{-1} V_2\phi_1\otimes V_1 \tilde \phi_1 \\
-a\lambda^{-1} V_2\tilde \phi_1 \otimes V_1 \phi_1 & \lambda^{-2} V_2P_0 V_1 + i \lambda^{-1} V_2P_0 V \frac{|x-y|^2}{24\pi} V P_0 V_1 - \lambda^{-1} a V_2\tilde \phi_1 \otimes V_1 \tilde \phi_1
\epm.
\ee
Therefore, if we define the canonical resonance as $\phi = \phi_1 - \tilde \phi_1$, $\phi$ satisfies $\phi \in \mc M$ and $\langle \phi, V \rangle = 1$ and
\be\lb{QQ'}
(Q(I+\widehat T(\lambda))Q)^{-1} = \frac{V_2 P_0 V_1} {\lambda^2} + \frac{i V_2 P_0 V \frac{|x-y|^2}{24\pi} V P_0 V_1} {\lambda} - \frac {a V_2\phi \otimes V_1\phi}{\lambda} + E(\lambda).
\ee

We apply Lemma \ref{lemma_invers} again after writing $I+\widehat T(\lambda)$ in matrix form with respect to the decomposition $L^2  = \ov Q L^2 + Q L^2$, where $Q L^2 = V_2 \mc M$:
$$
I + \widehat T(\lambda) = \bpm
\ov Q (I+\widehat T(\lambda)) \ov Q & \ov Q\widehat T(\lambda)Q \\
Q\widehat T(\lambda)\ov Q & Q (I+\widehat T(\lambda)) Q 
\epm := \bpm
S_{00}(\lambda)& S_{01}(\lambda)\\ 
S_{10}(\lambda)& S_{11}(\lambda)
\epm.
$$

Next, let $A(\lambda) := S_{00}(\lambda)^{-1}$. Then $\chi(\lambda/\epsilon)A(\lambda) \in \W$ for sufficiently small $\epsilon$. Indeed, 
it is easy to see that $S_{00}(\lambda) \in \W$. Furthermore, $S_{00}(0)$ is invertible on $\ov Q L^{3/2, 2} \cap \ov Q L^{3, 2}$ of inverse $K$, see (\ref{inverse}).

%
As in the proof of Lemma \ref{lemma24}, let $S_\epsilon(\lambda) = \chi(\lambda/\epsilon)\ov Q(\widehat T(\lambda) - \widehat T(0)) \ov Q$. A simple argument based on condition C1 shows that $\lim_{\epsilon \to 0} \|S_\epsilon(\lambda)\|_{\V_{L^{3/2, 2}} \cap \V_{L^{3, 2}}}=0$. Then
$$\begin{aligned}
\chi(\lambda/\epsilon) S_{00}^{-1}(\lambda) &= \chi(\lambda/\epsilon) \big(S_{00}(0) + \chi(\frac\lambda{2\epsilon}) \ov Q (\widehat T(\lambda) - \widehat T(0)) \ov Q\big)^{-1}\\
&=\chi(\lambda/\epsilon) S_{00}^{-1}(0) \sum_{k=0}^\infty (-1)^k (S_{2\epsilon}(\lambda) S_{00}^{-1}(0))^k.
\end{aligned}$$
This series converges for sufficiently small $\epsilon$, showing that $(\chi(\lambda/\epsilon) S_{00}^{-1}(\lambda))^\vee \in \V_{L^{3/2, 2}} \cap \V_{L^{3, 2}}$.

Concerning the derivative,
$$
\chi(\lambda/\epsilon) \partial_\lambda S_{00}^{-1}(\lambda) = - \chi(\lambda/\epsilon) S_{00}^{-1}(\lambda) \partial_{\lambda} S_{00}(\lambda) \chi(\frac \lambda {2\epsilon}) S_{00}^{-1}(\lambda).
$$
In this expression $(\chi(\lambda/\epsilon) S_{00}^{-1}(\lambda))^\vee \in \V_{L^{3/2, 2}} \cap \V_{L^{3, 2}}$ and $(\chi(\frac \lambda {2\epsilon}) \partial_{\lambda} S_{00}(\lambda))^\vee \in \V_{L^{3/2, 2}, L^{3, 2}}$ since $\ds M((\partial_{\lambda} T_{00}(\lambda))^\vee) = \frac{|V_2| \otimes |V_1|}{4\pi}$. Thus $(\chi(\lambda/\epsilon) \partial_\lambda S_{00}^{-1}(\lambda))^\vee \in \V_{L^{3/2, 2}, L^{3, 2}}$.

From this we infer that $\chi(\lambda/\epsilon) A(\lambda) \in \W$, so $A$ is a regular term.

We compute the inverse of $I + \widehat T(\lambda)$ by finding each of its matrix elements:
\be\lb{matrix}
(I+\widehat T(\lambda))^{-1} = \bpm A + A S_{01} C^{-1} S_{10} A & A S_{01} C^{-1} \\
-C^{-1} S_{10} A & C^{-1} \epm.
\ee
Here
$$
C(\lambda) = S_{11}(\lambda) - S_{10}(\lambda) A(\lambda) S_{01}(\lambda).
$$
$S_{10}(\lambda) A(\lambda) S_{01}(\lambda) = Q\widehat T(\lambda)A(\lambda)\widehat T(\lambda)Q$ may be written as 
\be\lb{expresie''}
\bpm
Q_1 \widehat T(\lambda)A(\lambda)\widehat T(\lambda)Q_1 &Q_1 \widehat T(\lambda)A(\lambda)\widehat T(\lambda)Q_2 \\
Q_2 \widehat T(\lambda)A(\lambda)\widehat T(\lambda)Q_1 &Q_2 \widehat T(\lambda)A(\lambda)\widehat T(\lambda)Q_2
\epm =
\bpm 
\lambda^2 E_{11}(\lambda) &\lambda^3 E_{12}(\lambda) \\
\lambda^3 E_{21} (\lambda) &\lambda^4 E_{22} (\lambda)
\epm.
\ee
Indeed, consider for example $Q_2 \widehat T(\lambda)A(\lambda)\widehat T(\lambda)Q_2$. It can be reexpressed as
\be\lb{expresie}\begin{aligned}
Q_2 \widehat T(\lambda)A(\lambda)\widehat T(\lambda)Q_2 =\lambda^4 Q_2 V_2 \frac {R_0((\lambda+i0)^2) - R_0(0) - i\lambda \frac {1\otimes 1}{4\pi}}{\lambda^2} V_1 A(\lambda) \\
V_2 \frac {{R_0((\lambda+i0)^2) - R_0(0) - i\lambda \frac {1\otimes 1}{4\pi}}} {\lambda^2} V_1 Q_2.
\end{aligned}\ee
For this computation we assume that $V \in \langle x \rangle^{-4} L^{3/2, 1}$. Taking a derivative of (\ref{expresie}) we obtain terms such as
\be\lb{expresie'}
Q_2 V_2 \partial_{\lambda} \Big(\frac {R_0((\lambda+i0)^2) - R_0(0) - i\lambda \frac {1\otimes 1}{4\pi}}{\lambda^2}\Big) V_1 A(\lambda) V_2 \frac {{R_0((\lambda+i0)^2) - R_0(0) - i\lambda \frac {1\otimes 1}{4\pi}}} {\lambda^2} V_1 Q_2.
\ee
Note that the range of $Q_2$ is spanned by functions $V_2 \phi_k$, $2 \leq k \leq N$, such that $|\phi_k(y)| \les \langle y \rangle^{-2}$ and $V_2 \in \langle x \rangle^{-2} L^{3, 2}$, so $V_2 \phi \in \langle y \rangle^{-4} L^{3, 2}$. Also
$$
M((V_2 \partial_{\lambda} \Big(\frac {R_0((\lambda+i0)^2) - R_0(0) - i\lambda \frac {1\otimes 1}{4\pi}}{\lambda^2}\Big) V_1)^\wedge) = |V_2| \frac{|x-y|^2}{24\pi} |V_1| \in \B(L^{3/2, 2}, L^{3, 2}).
$$
Likewise
$$
M((V_2 \frac {R_0((\lambda+i0)^2) - R_0(0) - i\lambda \frac {1\otimes 1}{4\pi}}{\lambda^2} V_1)^\wedge) = |V_2| \frac{|x-y|}{8\pi} |V_1| \in \B(L^{3/2, 2}, L^{3/2, 2}).
$$
This shows that $(\ref{expresie'}) \in \V_{L^{3/2, 2}, L^{3, 2}}$. By such computations we obtain that $Q_2 \widehat T(\lambda) A(\lambda) \widehat T(\lambda) Q_2 = \lambda^4 E_{22}(\lambda)$, where $\chi(\lambda/\epsilon) E_{22}(\lambda)\in\W$ for sufficiently small $\epsilon$. In this manner we prove (\ref{expresie''}).

By (\ref{QQ}), $S_{11}^{-1}(\lambda) = (Q\widehat T(\lambda)Q)^{-1}$ is of the form
$$
S_{11}^{-1}(\lambda) = \bpm
\lambda^{-1} E(\lambda) &\lambda^{-1} E(\lambda)\\
\lambda^{-1} E(\lambda) &\lambda^{-2} E(\lambda)
\epm.
$$
Then, letting $N(\lambda) := S_{11}^{-1}(\lambda) S_{10}(\lambda) A(\lambda) S_{01}(\lambda)$, by (\ref{expresie''})
$$\begin{aligned}
N(\lambda) &:= S_{11}^{-1}(\lambda) S_{10}(\lambda) S^{-1}_{00}(\lambda) S_{01}(\lambda) \\
&= \bpm
\lambda^{-1} E(\lambda) &\lambda^{-1} E(\lambda)\\
\lambda^{-1} E(\lambda) &\lambda^{-2} E(\lambda)
\epm
\bpm 
\lambda^2 E_{11}(\lambda) &\lambda^3 E_{12}(\lambda) \\
\lambda^3 E_{21} (\lambda) &\lambda^4 E_{22} (\lambda)
\epm\\
&= \bpm
\lambda E(\lambda) & \lambda^2 E(\lambda) \\
\lambda E(\lambda) & \lambda^2 E(\lambda)
\epm. 
\end{aligned}$$
This shows that $C(\lambda)$ is invertible for $\lambda \ll 1$:
$$
C(\lambda) = S_{11}(\lambda) - S_{10}(\lambda) A(\lambda) S_{01}(\lambda) = S_{11}(\lambda)(1 - N(\lambda)),
$$
so
\be\lb{C1}\begin{aligned}
C^{-1}(\lambda) &= (I - N(\lambda))^{-1} S_{11}^{-1}(\lambda) \\
&= S_{11}^{-1}(\lambda) + (I-N(\lambda))^{-1} N (\lambda) S_{11}^{-1}(\lambda).
\end{aligned}\ee
A computation shows that $(I-N(\lambda))^{-1} N (\lambda) S_{11}^{-1}(\lambda)$ is a regular term:
$$\begin{aligned}
(1-N(\lambda))^{-1} N (\lambda) S_{11}^{-1}(\lambda) &= E(\lambda) \bpm
\lambda E(\lambda) & \lambda^2 E(\lambda) \\
\lambda E(\lambda) & \lambda^2 E(\lambda)
\epm
\bpm
\lambda^{-1} E(\lambda) &\lambda^{-1} E(\lambda)\\
\lambda^{-1} E(\lambda) &\lambda^{-2} E(\lambda)
\epm \\
&= E(\lambda).
\end{aligned}$$
By (\ref{C1}) and (\ref{QQ'})
$$\begin{aligned}
C^{-1}(\lambda) &= S_{11}^{-1}(\lambda) + E(\lambda)\\
&= \lambda^{-2} V_2 P_0 V_1 + i \lambda^{-1} V_2 P_0 V \frac{|x-y|^2}{24} V P_0 V_1 - a \lambda^{-1} V_2 \phi \otimes V_1 \phi + E(\lambda).
\end{aligned}$$
One can then also write $C^{-1}$ as
$$
C^{-1}(\lambda) = \bpm
\lambda^{-1} E(\lambda) &\lambda^{-1} E(\lambda)\\
\lambda^{-1} E(\lambda) &\lambda^{-2} E(\lambda)
\epm.
$$
We also have
$$
S_{01}(\lambda) = \ov Q(I+\widehat T(\lambda))Q = \lambda E_1(\lambda) Q_1 + \lambda^2 E_2(\lambda) Q_2
$$
with regular terms $E_1, E_2 \in \W$:
$$\begin{aligned}
E_1(\lambda) &:= \ov Q V_2 \frac{R_0((\lambda+i0)^2)-R_0(0)} \lambda V_1 Q_1,\\
E_2(\lambda) &:= \ov Q V_2 \frac{R_0((\lambda+i0)^2)-R_0(0)-i\lambda 1 \otimes 1}{\lambda^2} V_1 Q_2.
\end{aligned}$$
Showing that $E_1$, $E_2 \in \W$ requires assuming that $V \in \langle x \rangle^{-4} L^{3/2, 1}$.

Therefore the following matrix element of (\ref{matrix}) is regular near zero:
$$\begin{aligned}
&A(\lambda)S_{01}(\lambda) C^{-1}(\lambda) =\\
&= (\lambda A(\lambda) E_1(\lambda)\ \lambda^2 A(\lambda) E_2(\lambda)) 
\bpm
\lambda^{-1} E(\lambda) &\lambda^{-1} E(\lambda) \\
\lambda^{-1} E(\lambda) &\lambda^{-2} E(\lambda)
\epm = E(\lambda).
\end{aligned}$$
One shows in the same manner that the matrix element $C^{-1}(\lambda) S_{10}(\lambda) A(\lambda)$ of (\ref{matrix}) is regular near zero.

Finally, the last remaining matrix element $A + A S_{01} C^{-1} S_{10} A$ of (\ref{matrix}) consists of the regular part $A$ and
$$\begin{aligned}
A S_{01} C^{-1} S_{10} A &= E(\lambda) (\lambda E(\lambda)\ \lambda^2 E(\lambda)) \bpm
\lambda^{-1} E(\lambda) &\lambda^{-1} E(\lambda) \\
\lambda^{-1} E(\lambda) &\lambda^{-2} E(\lambda)
\epm \bpm \lambda E(\lambda) \\ \lambda^2 E(\lambda) \epm
E(\lambda)\\
&= \lambda E(\lambda).
\end{aligned}$$
Thus this is also a regular term. It follows by (\ref{matrix}) that $\widehat T(\lambda)^{-1}$ is up to regular terms given by
$$
\lambda^{-2} V_2 P_0 V_1 + i \lambda^{-1} V_2 P_0 V \frac{|x-y|^2}{24} V P_0 V_1 - a \lambda^{-1} V_2 \phi \otimes V_1 \phi,
$$
which was to be shown.
\end{proof}

We next prove a corresponding statement in the case when $V$ has an almost minimal amount of decay. One can also obtain a resolvent expansion when $V \in \langle x \rangle^{-1} L^{3/2, 1}$, but that one does not lead to decay estimates.

\begin{lemma}\lb{lemma217} Suppose that $V \in \langle x \rangle^{-2} L^{3/2, 1}$ and $H=-\Delta+V$ is an exceptional Hamiltonian of the third kind. Let $\chi$ be a standard cutoff function. Then for sufficiently small $\epsilon$
$$\begin{aligned}
\chi(\lambda/\epsilon) (I + \widehat T(\lambda))^{-1} &= L(\lambda) + \lambda^{-1} S(\lambda) + \lambda^{-2}V_2 P_0 V_1,
\end{aligned}$$
where $L(\lambda) \in \W$, $S(\lambda)^\vee \in \V_{L^{3, 2}, L^{3/2, 2}}$, and $P_0$ is the $L^2$ orthogonal projection on $\mc E$.

Furthermore, $0$ is an isolated exceptional point, so $H$ has finitely many negative eigenvalues.
\end{lemma}
\begin{proof}[Proof of Lemma \ref{lemma217}]
We study $(I+\widehat T(\lambda))^{-1} := (I + V_2R_0((\lambda+i0)^2)V_1)^{-1}$ near $\lambda = 0$.

Let $Q=Q_1+Q_2$, $Q_0 = \ov Q$, $Q_1$, and $Q_2$ be as in the proof of Lemma \ref{lemma28}.

Also take again the orthonormal basis $\{\phi_1 , . . . , \phi_N\}$ with respect to the inner product $-(V u, v)$ for $\mc M$ so that $\{\phi_2 , \ldots , \phi_N\}$ is a basis of $\mc E$ and $\langle \phi_1 , V\rangle > 0$.
%

We apply Lemma \ref{lemma_invers} to invert $Q(I+\widehat T(\lambda))Q$ in $Q L^2$ for small $\lambda$, after writing it in matrix form with respect to the decomposition $QL^2 = Q_1 L^2 + Q_2 L^2$:
$$
Q(I+\widehat T(\lambda))Q = \bpm 
Q_1 (I+\widehat T(\lambda)) Q_1 &Q_1 \widehat T(\lambda) Q_2\\
Q_2 \widehat T(\lambda) Q_1 &Q_2 (I+\widehat T(\lambda)) Q_2
\epm := \bpm
T_{11}(\lambda) &T_{12}(\lambda)\\
T_{21}(\lambda) &T_{22}(\lambda)\epm.
$$
The inverse will be given by formula (\ref{invers}), that is
\be\lb{invers''}
(Q(I+\widehat T(\lambda))Q)^{-1} = \bpm T_{11}^{-1} + T_{11}^{-1} T_{12} C_{22}^{-1} T_{21} T_{11}^{-1}&-T_{11}^{-1} T_{12} C_{22}^{-1} \\ 
-C_{22}^{-1} T_{21} T_{11}^{-1} & C_{22}^{-1}
\epm,
\ee
where
$$
C_{22} = T_{22} - T_{21} T_{11}^{-1} T_{12}.
$$
Then (recall that $Q_1=-V_2\phi_1 \otimes V_1 \phi_1$)
$$\begin{aligned}
T_{11}(\lambda) &= Q_1 (I+\widehat T(\lambda)) Q_1 = Q_1 (I + V_2 R_0((\lambda+i0)^2) V_1) Q_1\\
&=Q_1 (V_2 R_0((\lambda+i0)^2) V_1 - V_2 R_0(0) V_1) Q_1\\
&=V_2 \phi_1 \otimes V\phi_1 (R_0((\lambda+i0)^2) - R_0(0)) V \phi_1 \otimes V_1 \phi_1 \\
&:= \lambda c_0(\lambda) Q_1.
\end{aligned}$$
Here $c_0(0)=\ds a=\frac {4i\pi}{|\langle V, \phi_1\rangle|^2} \ne 0$. Note that $c_0(\lambda) \in \widehat L^1$ when
$$
\int_{\R^3} \int_{\R^3} V(x) \phi_1(x)V(y) \phi_1(y) \Big\|\frac {e^{i\lambda|x-y|} - 1}{\lambda|x-y|}\Big\|_{\widehat L^1_\lambda} \dd x \dd y < \infty.
$$
Since $\Big\|\frac {e^{i\lambda|x-y|} - 1}{\lambda|x-y|}\Big\|_{\widehat L^1_\lambda} = 1$, it is enough to assume that $V \phi_1 \in L^1$, i.e.\ that $V \in L^{3/2, 1}$, in view of the fact that $\phi_1 \in \langle x \rangle^{-1} L^\infty$.

It follows that $T_{11}(\lambda)$ is invertible for $|\lambda| \ll 1$ in $Q_1 L^2$ and
$$
T_{11}^{-1}(\lambda) = \lambda^{-1} c_0^{-1} (\lambda) Q_1 = \lambda^{-1} E(\lambda).
$$
Here $\chi(\lambda/\epsilon) c_0^{-1}(\lambda) \in \widehat L^1$ for sufficiently small $\epsilon$.

Likewise, since $Q_2 (V_2 \otimes V_1) = (V_2 \otimes V_1) Q_2 = 0$,
$$\begin{aligned}
T_{12}(\lambda) &= Q_1 (I+V_2 R_0((\lambda+i0)^2) V_1) Q_2\\
&= \lambda^2 Q_1 V_2 \frac{R_0((\lambda+i0)^2) - R_0(0) - i \lambda (4\pi)^{-1} 1 \otimes 1}{\lambda^2} 
V_1 Q_2\\
&=\lambda^2 Q_1 e(\lambda) Q_2.
\end{aligned}$$
Since by Lemma \ref{fourier}
$$
M\Big(\Big(\frac{R_0((\lambda+i0)^2) - R_0(0) - i \lambda (4\pi)^{-1} 1 \otimes 1}{\lambda^2}\Big)^\wedge\Big) = \frac {|x-y|}{8\pi},
$$
it follows that $e(\lambda) \in \widehat L^1$ if
$$
\int_{\R^3} \int_{\R^3} V(x) \phi_1(x) V(y) \phi_k(y) |x-y| < \infty,
$$
that is if $V \in L^1$.


Likewise we obtain
$$\begin{aligned}
T_{21}(\lambda) &= \lambda^2 Q_2 e(\lambda) Q_1,
\end{aligned}$$
hence, combining the previous results,
$$\begin{aligned} 
T_{21}(\lambda) T_{11}^{-1}(\lambda) T_{12}(\lambda) &= \lambda^3 Q_2 e(\lambda) Q_2.
\end{aligned}$$
Furthermore
$$\begin{aligned}
T_{22}(\lambda) &= Q_2 (I+V_2 R_0((\lambda+i0)^2) V_1) Q_2 \\
&= \lambda^2 Q_2 V_2\frac{R_0((\lambda+i0)^2) - R_0(0) - i \lambda \frac{1 \otimes 1}{4\pi}}{\lambda^2} V_1 Q_2\\
&=-\lambda^2 (Q_2 V_2\frac{|x-y|}{8\pi} V_1 Q_2 + \lambda Q_2 e(\lambda) Q_2).
\end{aligned}$$
Again by Lemma \ref{fourier}, $e(\lambda) \in \widehat L^1$ if
$$
\int_{\R^3} \int_{\R^3} V(x) \phi_k(x) V(y) \phi_\ell(y) |x-y|^2 < \infty,
$$
that is (taking into account that $\phi_k$, $\phi_\ell \les \langle x \rangle^{-2}$) if $V \in L^1$.

Let $P_0$ be the $L^2$ orthogonal projection onto the set $\mc E$ spanned by $\phi_2, \ldots, \phi_N$. By relation (4.38) of \cite{yajima_disp},
$$
\Big(Q_2 V_2\frac{|x-y|}{8\pi} V_1 Q_2 \Big)^{-1} = -V_2P_0 V_1.
$$
Then
$$\begin{aligned}
C_{22}(\lambda) &= T_{22}(\lambda) - T_{21}(\lambda) T^{-1}_{11}(\lambda) T_{12}(\lambda) \\
&= -\lambda^2 Q_2 V_2\frac{|x-y|}{8\pi} V_1 Q_2 + \lambda^3 Q_2 e(\lambda) Q_2.
\end{aligned}$$
Therefore
$$
C_{22}^{-1}(\lambda) = \lambda^{-2} V_2 P_0 V_1 + \lambda^{-1} Q_2 e(\lambda) Q_2.
$$
Furthermore, we then obtain that
$$\begin{aligned}
-T_{11}^{-1}(\lambda) T_{12}(\lambda) C_{22}^{-1}(\lambda) &= \lambda^{-1} Q_1 e(\lambda) Q_1 \lambda^2 Q_1 e(\lambda) Q_2 \lambda^{-2} Q_2 e(\lambda) Q_2 \\
&= \lambda^{-1} Q_1 e(\lambda) Q_2.
\end{aligned}$$
Likewise we obtain
$$\begin{aligned}
-C_{22}^{-1}(\lambda) T_{21}(\lambda) T_{11}^{-1}(\lambda) &= \lambda^{-1} Q_2 e(\lambda) Q_1,\\
T_{11}^{-1}(\lambda) T_{12}(\lambda) C_{22}^{-1}(\lambda) T_{21}(\lambda) T^{-1}_{11}(\lambda) &= Q_1 e(\lambda) Q_1.
\end{aligned}$$
By (\ref{invers''}), $(Q(I+\widehat T(\lambda))Q)^{-1}$ is given in matrix form by
\be\begin{aligned}
\lb{QQQ}
(Q(I+\widehat T(\lambda))Q)^{-1} &= \bpm
\lambda^{-1} Q_1 e(\lambda) Q_1 & \lambda^{-1} Q_1 e(\lambda) Q_2 \\
\lambda^{-1} Q_2 e(\lambda) Q_1 & \lambda^{-2} V_2P_0 V_1 + \lambda^{-1} Q_2 e(\lambda) Q_2 \\
&= \lambda^{-1} Q e(\lambda) Q + \lambda^{-2} V_2 P_0 V_1.
\epm,
\end{aligned}\ee
where $\chi(\lambda/\epsilon) e(\lambda) \in \widehat L^1$ for sufficiently small $\epsilon$.

We apply Lemma \ref{lemma_invers} again after writing $I+\widehat T(\lambda)$ in matrix form with respect to the decomposition $L^2  = \ov Q L^2 + Q L^2$, where $Q L^2 = V_2 \mc M$:
$$
I + \widehat T(\lambda) = \bpm
\ov Q (I+\widehat T(\lambda)) \ov Q & \ov Q\widehat T(\lambda)Q \\
Q\widehat T(\lambda)\ov Q & Q (I+\widehat T(\lambda)) Q 
\epm := \bpm
S_{00}(\lambda)& S_{01}(\lambda)\\ 
S_{10}(\lambda)& S_{11}(\lambda)
\epm.
$$

Next, as in the proof of Lemma \ref{lemma28}, let $A(\lambda)=S_{00}^{-1}(\lambda)$. Then $\chi(\lambda/\epsilon)A(\lambda) \in \W$ for sufficiently small $\epsilon$.

We compute the inverse of $I + \widehat T(\lambda)$ by finding each of its matrix elements:
\be\lb{matrix2}
(I+\widehat T(\lambda))^{-1} = \bpm A + A S_{01} C^{-1} S_{10} A & A S_{01} C^{-1} \\
-C^{-1} S_{10} A & C^{-1} \epm.
\ee
Here
$$
C(\lambda) = S_{11}(\lambda) - S_{10}(\lambda) A(\lambda) S_{01}(\lambda).
$$
$S_{10}(\lambda) A(\lambda) S_{01}(\lambda) = Q\widehat T(\lambda)A(\lambda)\widehat T(\lambda)Q$ may be written as 
\be\lb{expresie''2}
\bpm
Q_1 \widehat T(\lambda)A(\lambda)\widehat T(\lambda)Q_1 &Q_1 \widehat T(\lambda)A(\lambda)\widehat T(\lambda)Q_2 \\
Q_2 \widehat T(\lambda)A(\lambda)\widehat T(\lambda)Q_1 &Q_2 \widehat T(\lambda)A(\lambda)\widehat T(\lambda)Q_2
\epm =
\bpm 
\lambda^2 Q_1 e(\lambda) Q_1 &\lambda^3 Q_1 e(\lambda) Q_2\\
\lambda^3 Q_2 e(\lambda) Q_1&\lambda^3 Q_2 e(\lambda) Q_2
\epm,
\ee
where $e(\lambda) \in \widehat L^1$.

Indeed, consider for example $Q_2 \widehat T(\lambda)A(\lambda)\widehat T(\lambda)Q_2$. It can be rewritten as
\be\lb{expresie2}\begin{aligned}
Q_2 \widehat T(\lambda)A(\lambda)\widehat T(\lambda)Q_2 =\lambda^3 Q_2 V_2 \frac {R_0((\lambda+i0)^2) - R_0(0) - i\lambda 1 \otimes 1}{\lambda} V_1 A(\lambda) \\
V_2 \frac {R_0((\lambda+i0)^2) - R_0(0)} {\lambda} V_1 Q_2.
\end{aligned}\ee
Assuming that $V \in \langle x \rangle^{-2} L^{3/2, 1}$
$$
M((V_2 \frac {R_0((\lambda+i0)^2) - R_0(0)}{\lambda} V_1)^\wedge) = \frac {|V_2| \otimes |V_1|}{4\pi} \in \B(L^{3/2, 2}).
$$
Likewise
$$
M((V_2 \frac {R_0((\lambda+i0)^2) - R_0(0) - i\lambda \frac {1\otimes 1}{4\pi}}{\lambda^2} V_1)^\wedge) = |V_2| \frac{|x-y|}{8\pi} |V_1| \in \B(L^{3/2, 2}, L^{3, 2}).
$$
This implies that $(\ref{expresie2}) = \lambda^3 Q_2 e_0(\lambda) Q_2$, $e_0(\lambda) \in \widehat L^1$. 
In this manner we prove (\ref{expresie''2}).


By (\ref{QQQ}), $S_{11}^{-1}(\lambda) = (Q\widehat T(\lambda)Q)^{-1}$ is of the form
$$
S_{11}^{-1}(\lambda) = \bpm
\lambda^{-1} Q_1 e(\lambda) Q_1 &\lambda^{-1} Q_1 e(\lambda) Q_2\\
\lambda^{-1} Q_2 e(\lambda) Q_1 &\lambda^{-2} Q_2 e(\lambda) Q_2
\epm.
$$
Then, letting $N(\lambda) := S_{11}^{-1}(\lambda) S_{10}(\lambda) A(\lambda) S_{01}(\lambda)$, by (\ref{expresie''2})
$$\begin{aligned}
N(\lambda) &:= S_{11}^{-1}(\lambda) S_{10}(\lambda) S^{-1}_{00}(\lambda) S_{01}(\lambda) \\
&= \bpm
\lambda^{-1} Q_1 e(\lambda) Q_1 &\lambda^{-1} Q_1 e(\lambda) Q_2\\
\lambda^{-1} Q_2 e(\lambda) Q_1 &\lambda^{-2} Q_2 e(\lambda) Q_2
\epm
\bpm 
\lambda^2 Q_1 e(\lambda) Q_1 &\lambda^3 Q_1 e(\lambda) Q_2\\
\lambda^3 Q_2 e(\lambda) Q_1 &\lambda^3 Q_2 e (\lambda) Q_2
\epm\\
&= \bpm
\lambda Q_1 e(\lambda) Q_1 & \lambda^2 Q_1 e(\lambda) Q_2 \\
\lambda Q_2 e(\lambda) Q_1 & \lambda Q_2 e(\lambda) Q_2
\epm. 
\end{aligned}$$
Therefore $N(0)=0$. This shows that $C(\lambda)$ is invertible for $\lambda \ll 1$:
$$
C(\lambda) = S_{11}(\lambda) - S_{10}(\lambda) A(\lambda) S_{01}(\lambda) = S_{11}(\lambda)(I - N(\lambda)),
$$
so
\be\lb{C2}\begin{aligned}
C^{-1}(\lambda) &= (I - N(\lambda))^{-1} S_{11}^{-1}(\lambda) \\
&= S_{11}^{-1}(\lambda) + (I-N(\lambda))^{-1} N (\lambda) S_{11}^{-1}(\lambda).
\end{aligned}\ee
A computation shows that
$$\begin{aligned}
(I-N(\lambda))^{-1} N (\lambda) S_{11}^{-1}(\lambda) &= Q e(\lambda) Q
\bpm
\lambda Q_1 e(\lambda) Q_1 & \lambda^2 Q_1 e(\lambda) Q_2 \\
\lambda Q_2 e(\lambda) Q_1 & \lambda Q_2 e(\lambda) Q_2
\epm \\
&\bpm
\lambda^{-1} Q_1 e(\lambda) Q_1 &\lambda^{-1} Q_1 e(\lambda) Q_2 \\
\lambda^{-1} Q_2 e(\lambda) Q_1 &\lambda^{-2} Q_2 e(\lambda) Q_2
\epm \\
&= \bpm
Q_1 e(\lambda) Q_1 & Q_1 e(\lambda) Q_2 \\
Q_2 e(\lambda) Q_1 & \lambda^{-1} Q_2 e(\lambda) Q_2
\epm.
\end{aligned}$$
By (\ref{C2}) and (\ref{QQQ})
$$\begin{aligned}
C^{-1}(\lambda) &= S_{11}^{-1}(\lambda) + \lambda^{-1} Q e(\lambda) Q\\
&= \lambda^{-2} V_2 P_0 V_1 + \lambda^{-1} Q e(\lambda) Q.
\end{aligned}$$
Note that
$$\begin{aligned}
S_{01}(\lambda)&=\ov Q \widehat T(\lambda) Q = \ov Q(I+\widehat T(\lambda)) Q \\
&= \lambda \ov Q V_2 \frac {R_0((\lambda+i0)^2)-R_0(0)} \lambda V_1 Q = \lambda E_1(\lambda),
\end{aligned}$$
where $E_1(\lambda)^\vee \in \V_{L^{3/2, 2}}$ when $V \in \langle x \rangle^{-1} L^1$.
Therefore
$$
A(\lambda) S_{01}(\lambda) C^{-1}(\lambda) = A(\lambda) \lambda  E_1(\lambda) \lambda^{-2} Q e(\lambda) Q = \lambda^{-1} S(\lambda),
$$
where $S(\lambda)^\vee \in \V_{L^{3, 2}, L^{3/2, 2}}$. Likewise $S_{10}(\lambda) = \lambda E_2(\lambda)$, where $E_2(\lambda)^\vee \in \V_{L^{3, 2}}$. Then
$$
C^{-1}(\lambda) S_{10}(\lambda) A(\lambda) = \lambda^{-1} S(\lambda),
$$
where $S(\lambda)^\vee \in \V_{L^{3, 2}, L^{3/2, 2}}$.

Finally, for the last remaining matrix element $A + A S_{01} C^{-1} S_{10} A$ of (\ref{matrix2}) we use the fact that
$$
A S_{01} C^{-1} S_{10} A = A(\lambda) \lambda E_1(\lambda) \lambda^{-2} Q e(\lambda) Q \lambda E_2(\lambda) A(\lambda) = S(\lambda),
$$
where $S(\lambda)^\vee \in \V_{L^{3, 2}, L^{3/2, 2}}$. 
%
Also recall that $A(\lambda) \in \W$.

We have thus analyzed all the terms in (\ref{matrix2}) and the conclusion follows.
\end{proof}

Recall that
$$\begin{aligned}
&R(t) := \frac{ae^{-i\frac{3\pi}4}}{\sqrt{\pi t}} \zeta_t(x) \otimes \zeta_t(y),\ \zeta_t(x) := e^{i|x|^2/4t} \phi(x),\\
&S(t) := \frac{e^{-i\frac{3\pi}4}}{\sqrt{\pi t}}\Big(-i P_0 V \frac{|x-y|^2}{24\pi} V P_0 + \mu_t(x) \frac{|x-y|}{8\pi} V P_0 + P_0 V \frac{|x-y|}{8\pi} \mu_t(y)\Big),
\end{aligned}$$
where $\mu_t(x) := \frac i {|x|} \int_0^1 (e^{i\frac{|x|^2}{4t}}-e^{i\frac{|\theta x|^2}{4t}}) \dd \theta$.

Although it is not immediately obvious, it is also true that
\be\lb{s_bound}
\|S(t)u\|_{L^{3, \infty}} \les t^{-1/2} \|u\|_{L^{3/2, 1}}.
\ee
Indeed, note that since $\langle \phi_k, V \rangle=0$ for the eigenvectors $\phi_k$, $2 \leq k \leq N$ (recall that $\phi_1$ is the resonance),
$$
\mu_t(x) |x-y| V P_0 = \mu_t(x) (|x-y| - |x|) V P_0,
$$
which is bounded in absolute value by $\sum_{k=2}^N |\mu_t(x)| \int_{\R^3} |y| |V(y)| |\phi_k(y)| \dd y \otimes |\phi_k(z)|$. By definition, $|\mu_t(x)| \les |x|^{-1}$. This leads to (\ref{s_bound}), since $\phi_k \in \langle x \rangle^{-2} L^\infty$ and $V \in L^{3/2, 1}$.

We use Lemma \ref{lemma28} as the basis for the following decay estimate:
\begin{proposition}\lb{prop_3} Let V satisfy $\langle x \rangle^4 V(x) \in L^{3/2, 1}$. Suppose that $H$ is of exceptional type of the third kind. Then, for $1 \leq p < 3/2$ and $u \in L^2 \cap L^p$, 
\be\lb{concluzie}
e^{-itH} P_cu=Z(t)u+R(t)u+S(t)u,\ \|Z(t) u\|_{L^{p'}} \les t^{-\frac 3 2(\frac 1 p - \frac 1 {p'})} \|u\|_{L^p}.
\ee
Here $\frac 1 p + \frac 1 {p'} = 1$. If in addition all the zero energy eigenfunctions $\phi_k$, $2 \leq k \leq N$ are in $L^1$, then we can take $S(t)=0$.
\end{proposition}
\begin{proof}[Proof of Proposition \ref{prop_3}] Write the dispersive component of the evolution as
$$
e^{itH}P_c f = \frac 1 {i\pi} \int_\R e^{it\lambda^2} \big(R_0((\lambda+i0)^2) - R_0((\lambda+i0)^2) V_1 \widehat T(\lambda)^{-1}  V_2 R_0((\lambda+i0)^2)\big) f\, \lambda \dd \lambda.
$$
We use the same method as in the proofs of Proposition \ref{prop28} and \ref{prop29}. Consider a partition of unity subordinated to the neighborhoods of Lemmas \ref{lemma23} and \ref{lemma28}. First, following Lemma \ref{lemma23}, take a sufficiently large $R$ such that $(1-\chi(\lambda/R)) (I+\widehat T(\lambda))^{-1} \in \W$. Then, again by Lemma \ref{lemma23}, for every $\lambda_0 \in [-4R, 4R]$ there exists $\epsilon(\lambda_0)>0$ such that $\chi(\frac{\lambda-\lambda_0}{\epsilon(\lambda_0)}(I+\widehat T(\lambda))^{-1} \in \W$ if $\lambda_0\ne 0$ or the conclusion of Lemma \ref{lemma28} holds when $\lambda_0=0$.

Since $[-4R, 4R]$ is a compact set, there exists a finite covering $[-4R, 4R] \subset \bigcup_{k=1}^N (\lambda_k - \epsilon(\lambda_k), \lambda_k + \epsilon(\lambda_k))$. Then we construct a finite partition of unity on $\R$ by smooth functions $1= \chi_0(\lambda) + \sum_{k=1}^N \chi_k(\lambda) + \chi_\infty(\lambda)$, where $\supp \chi_\infty \subset \R \setminus (-2R, 2R)$, $\supp \chi_0 \subset [-\epsilon(0), \epsilon(0)]$, and $\supp \chi_k \subset [\lambda_k-\epsilon(\lambda_k), \lambda_k+\epsilon(\lambda_k)]$.

By Lemma \ref{lemma23}, for any $k \ne 0$, $\chi_k(\lambda) (I+\widehat T (\lambda))^{-1} \in \W$, so $(1-\chi_0(\lambda)) (I+\widehat T(\lambda))^{-1} \in \W$. By Lemma \ref{lemma28}, for $L \in \W$
$$
\chi_0(\lambda)(I + \widehat T(\lambda))^{-1} = L(\lambda) + \chi_0(\lambda) \Big(\frac{V_2 P_0 V_1}{\lambda^2} + \frac{i V_2 P_0 V |x-y|^2 V P_0 V_1}{\lambda} - \frac a {\lambda} V_2 \phi \otimes V_1\phi\Big).
$$
Let $Z_1$ be the contribution of all the regular terms in this decomposition, such as the free resolvent, $(1-\chi_0(\lambda))(I+\widehat T(\lambda))^{-1}$, and $L(\lambda)$:
$$\begin{aligned}
Z_1(t)&:=\frac 1 {i\pi} \int_\R e^{-it\lambda^2} \big( R_0((\lambda+i0)^2) - R_0((\lambda+i0)^2) V_1 L(\lambda) V_2 R_0((\lambda+i0)^2)-\\
&-(1-\chi_0(\lambda))R_0((\lambda+i0)^2) V_1 \widehat T(\lambda) V_2 R_0((\lambda+i0)^2) \big) \lambda \dd\lambda \\
&=\frac 1 {2\pi t} \int_\R e^{-it\lambda^2} \partial_\lambda \big( R_0((\lambda+i0)^2) - R_0((\lambda+i0)^2) V_1 L(\lambda) V_2 R_0((\lambda+i0)^2) - \\
&-(1-\chi_0(\lambda))R_0((\lambda+i0)^2) V_1 \widehat T(\lambda) V_2 R_0((\lambda+i0)^2)\big) \dd \lambda \\
&=\frac C {t^{3/2}} \int_\R e^{-i\frac{\rho^2}{4t}} \big(\partial_\lambda \big( R_0((\lambda+i0)^2) - R_0((\lambda+i0)^2) V_1 L(\lambda) V_2 R_0((\lambda+i0)^2) - \\
&-(1-\chi_0(\lambda))R_0((\lambda+i0)^2) V_1 \widehat T(\lambda) V_2 R_0((\lambda+i0)^2)\big)\big)^\vee(\rho) \dd \rho.
\end{aligned}$$
The fact that $\|Z_1(t) u\|_{L^1} \les |t|^{-3/2} \|u\|_{L^\infty}$ follows by knowing that
$$\begin{aligned}
\big(\partial_\lambda \big( R_0((\lambda+i0)^2) - R_0((\lambda+i0)^2) V_1 L(\lambda) V_2 R_0((\lambda+i0)^2) - \\
-(1-\chi_0(\lambda))R_0((\lambda+i0)^2) V_1 \widehat T(\lambda) V_2 R_0((\lambda+i0)^2)\big)\big)^\vee \in \V_{L^1, L^\infty}.
\end{aligned}$$
By smoothing estimates it also follows that $Z_1(t)$ is $L^2$-bounded, see the proof of Proposition \ref{prop28}. By interpolation we also obtain the estimate $\|Z_1(t) u\|_{L^{3, \infty}} \les \|u\|_{L^{3/2, 1}}$.

Let $Z_2(t)$ be the contribution of the term $a \lambda^{-1} \chi_0(\lambda) V_2 \phi \otimes V_1\phi$:
$$
Z_2(t):= \frac a {i\pi} \int_\R e^{-it\lambda^2} \chi_0(\lambda) R_0((\lambda+i0)^2) V\phi \otimes V\phi R_0((\lambda+i0)^2) \dd \lambda.
$$
By Lemma \ref{lemma4.10}
$$
\|(Z_2(t)-R(t))u\|_{L^\infty} \leq t^{-3/2} \|u\|_{L^1},\ \|Z_2(t)u\|_{L^{3, \infty}} \les t^{-1/2} \|u\|_{L^{3/2, 1}}.
$$

We are left with the terms $\lambda^{-2} R_0((\lambda+i0)^2) V P_0 V R_0((\lambda+i0)^2)$ and $i\lambda^{-1} R_0((\lambda+i0)^2) V P_0 V \frac{|x-y|^2}{24\pi} V P_0 V R_0((\lambda+i0)^2)$. 
Let their contributions be
$$\begin{aligned}
X_2(t)&:=\frac {-1}{\pi} \int_{\R} e^{-it\lambda^2} R_0((\lambda+i0)^2) V P_0 V \frac{|x-y|^2}{24\pi} V P_0 V R_0((\lambda+i0)^2) \dd \lambda,\\
X_3(t)&:=\frac{-1}{i\pi}\lim_{\delta \to 0} \int_{|\lambda|>\delta} e^{-it\lambda^2} R_0((\lambda+i0)^2) V P_0 V R_0((\lambda+i0)^2) \lambda^{-1} \dd \lambda.
\end{aligned}$$
By Lemma 4.12 of \cite{yajima_disp},
\be\lb{rez_1}\begin{aligned}
\|X_2(t) u\|_{L^{3, \infty}} &\les t^{-1/2} \|u\|_{L^{3/2, 1}},\\
\Big\|X_2(t) u + i \frac {e^{-i \frac {3\pi}4}}{\sqrt {\pi t}} P_0 V \frac {|x-y|^2}{24\pi} V P_0\Big\|_{L^\infty} &\les t^{-3/2} \|u\|_{L^1}.
\end{aligned}\ee
This lemma has a proof similar to Lemma \ref{lemma4.10}. It requires, in addition, that $|\phi_j(x)| \les |x|^{-2}$ for every eigenfunction $\phi_j \in \mc E$, $2 \leq j \leq N$, which is guaranteed by Lemma \ref{lemma_12}.

By Lemma 4.14 of \cite{yajima_disp},
\be\lb{rez_2}\begin{aligned}
\|X_3(t) u\|_{L^{3, \infty}} &\les t^{-1/2} \|u\|_{L^{3/2, 1}},\\
\Big\|X_3(t) u - \frac {e^{-i \frac {3\pi}4}}{\sqrt {\pi t}} \Big(\mu_t(x) \frac{|x-y|}{8\pi} V P_0 + P_0 V \frac{|x-y|}{8\pi} \mu_t(y) \Big)\Big\|_{L^\infty} &\les t^{-3/2} \|u\|_{L^1}.
\end{aligned}\ee
The proof of Lemma 4.14 in \cite{yajima_disp} depends on $\langle y \rangle^3 V(y) \phi(y)$ being integrable, which is also true here since $|\phi(y)| \les \langle y \rangle^{-1}$ and $\langle y \rangle^2 V(y) \in \langle y \rangle^{-2} L^{3/2, 1} \subset L^1$.

Combining the two results (\ref{rez_1}) and (\ref{rez_2}) and knowing that $\|S(t)u\|_{L^{3, \infty}} \les t^{-1/2} \|u\|_{L^{3/2, 1}}$ by (\ref{s_bound}), we obtain that
\be\lb{s}
\|(X_2(t)+X_3(t)-S(t))u\|_{L^\infty} \les t^{-3/2} \|u\|_{L^1},\ \|(X_2(t)+X_3(t)-S(t))u\|_{L^{3, \infty}} \les t^{-1/2} \|u\|_{L^{3/2, 1}}.
\ee

Recall that $e^{-itH}P_c = Z_1(t)+Z_2(t)+X_2(t)+X_3(t) = Z(t)+R(t)+S(t)$. We obtain for $Z(t) = Z_1(t) + (Z_2(t) - R(t)) + (X_2(t) + X_3(t) - S(t))$ that
$$
\|Z(t)u\|_{L^\infty} \les t^{-3/2} \|u\|_{L^1},\ \|Z(t)u\|_{L^{3, \infty}} \les t^{-1/2} \|u\|_{L^{3/2, 1}}.
$$
Conclusion (\ref{concluzie}) follows by interpolation.

Finally, assume that all the eigenfunctions $\phi_k \in L^1$ for $2 \leq k \leq N$ (recall that $\phi_1$ is the resonance). Then, by Lemma \ref{decay_lemma}, it follows that $\langle V \phi_k, y_\ell \rangle = \langle V \phi_k, y_\ell y_m = 0$ for all $\ell$ and $m$ and all $2 \leq k \leq N$. As a consequence, we immediately see that
$$
P_0 V |x-y|^2 V P_0 = P_0 V (|x|^2 + |y|^2) P_0 - 2 \sum_{k=1}^3 P_0 V x_k y_k V P_0 = 0.
$$
Since $\langle \phi_k, V \rangle = 0$ and $\langle V \phi_k, y_\ell \rangle = 0$, we can also rewrite
$$
\mu_t(x) |x-y| V P_0 = \mu_t(x) (|x-y|-|x|+\frac {xy}{|x|}) V P_0.
$$
Then note that $|x| (|x-y|-|x|+\frac {xy}{|x|}) V P_0$ is bounded in absolute value by $\sum_{k=2}^N \int_{\R^3} |y|^2 |V(y)| |\phi_k(y)| \dd y \otimes |\phi_k(z)|$, which is bounded from $L^1$ to $L^\infty$ since $\phi_k \in \langle x \rangle^{-2} L^\infty$ and $V \in \langle x \rangle^{-1} L^{3/2, 1}$. Having gained a power of decay in $x$, we use it by $|\mu_t(x) |x|^{-1}| \les t^{-1}$. Therefore
$$
\|t^{-1/2} \mu_t(x) |x-y| V P_0 u\|_{L^\infty} \les t^{-3/2} \|u\|_{L^1}.
$$

Consequently, when $\phi_k \in L^1$ for $2 \leq k \leq N$, $S(t)$ can be removed from (\ref{s}). Hence we retrieve conclusion (\ref{concluzie}) without $S$, as claimed.
\end{proof}

\begin{proposition}\lb{prop219} Assume that $V \in \langle x \rangle^{-2} L^{3/2, 1}$ and that $H=-\Delta+V$ is an exceptional Hamiltonian of the third kind. Then
$$
\|e^{-itH}P_c u\|_{L^{3, \infty}} \les t^{-1/2} \|u\|_{L^{3/2, 1}}
$$
and for $3/2<p\leq 2$
$$
\|e^{-itH} P_c u\|_{L^{p'}} \les t^{-\frac 3 2(\frac 1 p - \frac 1 {p'})} \|u\|_{L^p}.
$$
Here $\frac 1 p + \frac 1 {p'} = 1$.
\end{proposition}
The proof of this proposition parallels the proof of Proposition \ref{prop29}.
\begin{proof}[Proof of Proposition \ref{prop219}]
Write the evolution as
$$
e^{-itH}P_c f = \frac 1 {i\pi} \int_\R e^{-it\lambda^2} \big(R_0((\lambda+i0)^2) - R_0((\lambda+i0)^2) V_1 \widehat T(\lambda)^{-1}  V_2 R_0((\lambda+i0)^2)\big) f \lambda \dd \lambda.
$$
We consider a partition of unity subordinated to the neighborhoods of Lemmas \ref{lemma23} and \ref{lemma217}. First, take a sufficiently large $R$ such that $(1-\chi(\lambda/R)) (I+\widehat T(\lambda))^{-1} \in \W$. Then, for every $\lambda_0 \in [-4R, 4R]$ there exists $\epsilon(\lambda_0)>0$ such that $\chi(\frac{\lambda-\lambda_0}{\epsilon(\lambda_0)}(I+\widehat T(\lambda))^{-1} \in \W$ if $\lambda_0\ne 0$ or the conclusion of Lemma \ref{lemma27} holds when $\lambda_0=0$.

Since $[-4R, 4R]$ is a compact set, there exists a finite covering $[-4R, 4R] \subset \bigcup_{k=1}^N (\lambda_k - \epsilon(\lambda_k), \lambda_k + \epsilon(\lambda_k))$. Then we construct a finite partition of unity on $\R$ by smooth functions $1= \chi_0(\lambda) + \sum_{k=1}^N \chi_k(\lambda) + \chi_\infty(\lambda)$, where $\supp \chi_\infty \subset \R \setminus (-2R, 2R)$, $\supp \chi_0 \subset [-\epsilon(0), \epsilon(0)]$, and $\supp \chi_k \subset [\lambda_k-\epsilon(\lambda_k), \lambda_k+\epsilon(\lambda_k)]$.

By Lemma \ref{lemma23}, for any $k \ne 0$, $\chi_k(\lambda) (I+\widehat T (\lambda))^{-1} \in \W$, so $(1-\chi_0(\lambda)) (I+\widehat T(\lambda))^{-1} \in \W$. By Lemma \ref{lemma217}
$$
\chi_0(\lambda) (I+\widehat T(\lambda))^{-1} = L(\lambda) + \lambda^{-1} S(\lambda) + \lambda^{-2} V_2 P_0 V_1,
$$
where $L \in \W$ and $S^\vee \in \V_{L^{3, 2}, L^{3/2, 2}}$.

Let $Z_1$ be given by the sum of all the regular terms of the decomposition:
$$\begin{aligned}
Z_1(t) &:= \frac 1 {i\pi} \int_\R e^{-it\lambda^2} \big(R_0((\lambda+i0)^2) - R_0((\lambda+i0)^2) V_1 L(\lambda) V_2 R_0((\lambda+i0)^2) - \\
&- (1-\chi_0(\lambda)) R_0((\lambda+i0)^2) V_1 \widehat T(\lambda) V_2 R_0((\lambda+i0)^2) \big) \lambda \dd \lambda \\
&= \frac 1 {2\pi t} \int_\R e^{-it\lambda^2} \partial_\lambda \big(R_0((\lambda+i0)^2) - R_0((\lambda+i0)^2) V_1 L(\lambda) V_2 R_0((\lambda+i0)^2) - \\
&- (1-\chi_0(\lambda)) R_0((\lambda+i0)^2) V_1 \widehat T(\lambda) V_2 R_0((\lambda+i0)^2) \big) \dd \lambda\\
&= \frac C {t^{3/2}} \int_\R e^{-i\frac{\rho^2}{4t}} \big(\partial_\lambda \big(R_0((\lambda+i0)^2) - R_0((\lambda+i0)^2) V_1 L(\lambda) V_2 R_0((\lambda+i0)^2) - \\
&- (1-\chi_0(\lambda)) R_0((\lambda+i0)^2) V_1 \widehat T(\lambda) V_2 R_0((\lambda+i0)^2) \big)\big)^\vee(\rho) \dd \rho.
\end{aligned}$$
The fact that $\|Z_1(t) u\|_{L^\infty} \les |t|^{-3/2} \|u\|_{L^1}$ follows by knowing that
$$\begin{aligned}
\big(\partial_\lambda \big(&R_0((\lambda+i0)^2) - R_0((\lambda+i0)^2) V_1 L(\lambda) V_2 R_0((\lambda+i0)^2) - \\
&- (1-\chi_0(\lambda)) R_0((\lambda+i0)^2) V_1 \widehat T(\lambda) V_2 R_0((\lambda+i0)^2)\big)\big)^\vee \in \V_{L^1, L^{\infty}}.
\end{aligned}$$

By smoothing estimates it immediately follows that $Z_1(t)$ is $L^2$-bounded, see the proof of Proposition \ref{prop28}. Interpolating, we obtain that $\|Z_1(t)u\|_{L^{3, \infty}} \les t^{-1/2} \|u\|_{L^{3, 1}}$.

Let $Z_2$ be the following singular term in the decomposition of Lemma~\ref{lemma217}:
$$\begin{aligned}
Z_2(t) &:= \frac{1}{i\pi} \int_{\R} e^{-it\lambda^2} R_0((\lambda+i0)^2) V_1 S(\lambda) V_2 R_0((\lambda+i0)^2) \dd \lambda \\
&= \frac C {t^{1/2}} \int_\R e^{-i\frac {\rho^2}{4t}} \big(R_0((\lambda+i0)^2) V_1 S(\lambda) V_2 R_0((\lambda+i0)^2)\big)^\vee(\rho) \dd \rho.
\end{aligned}$$
Note that $(R_0((\lambda+i0)^2) V_1)^\vee \in \V_{L^{3/2, 2}, L^{3, \infty}}$, $S(\lambda)^\vee \in \V_{L^{3, 2}, L^{3/2, 2}}$, and $(V_2 R_0((\lambda+i0)^2))^\vee \in \V_{L^{3/2, 1}, L^{3, 2}}$. Thus
$$
R_0((\lambda+i0)^2) V_1 (\lambda S(\lambda)) V_2 R_0((\lambda+i0)^2) \in \V_{L^{3/2, 1}, L^{3, \infty}}.
$$
This immediately implies that $\|Z_2(t) u\|_{L^{3, \infty}} \les t^{-1/2} \|u\|_{L^{3/2, 1}}$.

We are left with the contribution of the term $\lambda^{-2} V_2 P_0 V_1$. This is the same as the term $X_3$ from the proof of Proposition \ref{prop_3}. By (\ref{rez_2}) $\|X_3(t) u\|_{L^{3, \infty}} \les t^{-1/2} \|u\|_{L^{3/2, 1}}$.

Putting the three estimates for $Z_1$, $Z_2$, and $X_3$ together, we obtain that $\|e^{-itH}P_c u\|_{L^{3, \infty}} \les \|u\|_{L^{3/2, 1}}$. Interpolating with the obvious $L^2$ bound $\|e^{-itH} P_c u\|_{L^2} \les \|u\|_{L^2}$, we obtain the stated conclusion.
\end{proof}

\section*{Acknowledgments} This paper is indebted to Professor Kenji Yajima's article \cite{yajima_disp}, from which it borrows several lemmas, as well as the general plan of the proof. I would like to thank Professor Avy Soffer for referring me to that article and Professor Wilhelm Schlag for introducing me to this problem. I would also like to thank the anonymous referee for the many useful comments.

The author was partially supported by the NSF grant DMS--1128155 and by an AMS--Simons Foundation travel grant.


\begin{thebibliography}{AbcDef1}



\bibitem[Bec]{bec} M.\ Beceanu, \emph{New estimates for a time-dependent Schr\"{o}dinger equation}, Duke Math. J., Vol.\ 159, 3 (2011), pp.\ 417--477.

\bibitem[BeGo1]{becgol} M.\ Beceanu, M.\ Goldberg, \emph{Schr\"{o}dinger dispersive estimates for a scaling-critical class of potentials}, Communications in Mathematical Physics (2012), Vol.\ 314, Issue 2, pp.\ 471--481.

\bibitem[BeGo2]{becgol2} M.\ Beceanu, M.\ Goldberg, \emph{Strichartz Estimates and Maximal Operators for the Wave Equation in $\R^3$}, Journal of Functional Analysis (2014), Vol.\ 266, Issue 3, pp. 1476--1510.



\bibitem[BeL\"o]{bergh} J.\ Bergh, J.\ L\"ofstr\"om, \emph{Interpolation Spaces. An Introduction}, Springer-Verlag, 1976.





\bibitem[CCV]{ccv} F.\ Cardosa, C.\ Cuevas, G.\ Vodev, \emph{Dispersive estimates for the Schr\"{o}dinger equation in dimensions four and five}, Asymptot.\ Anal.\ (2009) 62(3Ð4), pp.\ 125--145.

\bibitem[EKMT]{ekmt} I.\ Egorova, E.\ Kopylova, V.\ Marchenko, G.\ Teschl, \emph{Dispersion estimates for one-dimensional Schr\"{o}dinger and Klein-Gordon equations revisited}, preprint, arXiv:1411.0021.

\bibitem[EGG]{egg} M.\ B.\ Erdogan, M.\ Goldberg, W.\ R.\ Green, \emph{Dispersive Estimates for Four Dimensional Schr\"{o}dinger and Wave Equations with Obstructions at Zero Energy}, Communications in Partial Differential Equations (2014), Vol. 39, Issue 10, pp.\ 1936--1964.

\bibitem[ErGr1]{ergr1} M.\ B.\ Erdogan, W.\ R.\ Green, \emph{Dispersive estimates for the Schr\"{o}dinger equation for $C^{\frac{n-3}2}$ potentials in odd dimensions}, Int.\ Math.\ Res.\ Notices (2010) (13), pp.\ 2532--2565.

\bibitem[ErGr2]{ergr2} M.\ B.\ Erdogan, W.\ R.\ Green, \emph{Dispersive estimates for Schr\"{o}dinger operators in dimension two with obstructions at zero energy}, Trans.\ Amer.\ Math.\ Soc.\ (2013), 365, pp.\ 6403--6440.

\bibitem[ErGr3]{ergr3} M.\ B.\ Erdogan, W.\ R.\ Green, \emph{A weighted dispersive estimate for Schr\"{o}dinger operators in dimension two}, Communications in Mathematical Physics (2013), Vol.\ 319, Issue 3, pp.\ 791--811.

\bibitem[ErGr4]{ergr4} M.\ B.\ Erdogan, W.\ R.\ Green, \emph{Dispersive estimates for matrix Schr\"{o}dinger operators in dimension two} (2013), Discrete and Continuous Dynamical Systems - Series A, Vol.\ 33, Issue 10, pp.\ 4473 -- 4495.

\bibitem[ErSc1]{ersc1} M.\ B.\ Erdogan,\ W.\ Schlag, \emph{Dispersive estimates for Schrodinger operators in the presence of a resonance and/or an eigenvalue at zero energy in dimension three: I}, Dyn.\ Partial Differ.\ Equ.\ 1 (2004), no.\ 4, pp.\ 359--379.

\bibitem[ErSc2]{ersc2} M.\ B.\ Erdogan, W.\ Schlag, \emph{Dispersive estimates for Schr\"odinger operators in the presence of a resonance and/or an eigenvalue at zero energy in dimension three: II}, Journal d'Analyse Math\'{e}matique, Vol.\ 99, No.\ 1 (2006), pp.\ 199--248.

\bibitem[Gol1]{gol} M.\ Goldberg, \emph{Dispersive bounds for the three-dimensional Schr\"{o}dinger equation with almost critical potentials}, Geom.\ and Funct.\ Anal.\ 16 (2006), no.\ 3, pp.\ 517--536.

\bibitem[Gol2]{gol3} M.\ Goldberg, \emph{Transport in the one-dimensional Schr\"{o}dinger equation}, Proc.\ Amer.\ Math.\ Soc. 135 (2007), pp.\ 3171--3179.

\bibitem[Gol3]{gol2} M.\ Goldberg, \emph{A dispersive bound for three-dimensional Schr\"{o}dinger operators with zero energy eigenvalues}, Communications in Partial Differential Equations (2010), Vol.\ 35, Issue 9, pp.\ 1610--1634.


\bibitem[GoGr1]{gogr1} M.\ Goldberg, W.\ R.\ Green, \emph{Dispersive estimates for higher dimensional Schr\"{o}dinger operators with threshold eigenvalues I: the odd dimensional case}, preprint, arXiv: 1409.6323.

\bibitem[GoGr2]{gogr2} M.\ Goldberg, W.\ R.\ Green, \emph{Dispersive estimates for higher dimensional Schr\"{o}dinger operators with threshold eigenvalues II: The even dimensional case}, preprint, arXiv: 1409.6328.

\bibitem[GoSc1]{golsch} M.\ Goldberg, W.\ Schlag, \emph{A limiting absorption principle for the three-dimensional Schr\"odinger equation with $L^p$ potentials}, Intl.\ Math.\ Res.\ Not.\ 2004:75 (2004), pp.\ 4049--4071.

\bibitem[GoSc2]{gosc} M.\ Goldberg, W.\ Schlag, \emph{Dispersive estimates for Schr\"{o}dinger operators in dimensions one and three}, Comm.\ Math.\ Phys.\ vol.\ 251, no.\ 1 (2004), pp.\ 157--178.

\bibitem[Gre]{green} W.\ R.\ Green, \emph{Dispersive estimates for matrix and scalar Schr\"{o}dinger operators in dimension five}, Illinois J.\ Math.\ (2012), Vol.\ 56, No.\ 2, pp.\ 307--341.

\bibitem[IoJe]{ionjer} A.\ D.\ Ionescu, D.\ Jerison, \emph{On the absence of positive eigenvalues of Schr\"odinger operators with rough potentials}, Geometric and Functional Analysis 13, pp.\ 1029--1081 (2003).



\bibitem[JeKa]{jeka} A.\ Jensen, T.\ Kato, \emph{Spectral properties of Schr\"{o}dinger operators and time--decay of the wave functions}. Duke Math.\ J.\ 46 (1979), no.\ 3, pp.\ 583--611.

\bibitem[JeNe]{jene} A.\ Jensen, G. Nenciu, \emph{A unified approach to resolvent expansions at thresholds}, Rev.\ Math.\ Phys. 13 (2001), no.\ 6, pp.\ 717--754.

\bibitem[JSS]{jss} J.-L.\ Journ\'{e}, A.\ Soffer, C.\ D.\ Sogge, \emph{Decay estimates for Schr\"{o}dinger operators}, Comm.\ Pure Appl.\ Math.\ 44 (1991), no.\ 5, pp.\ 573--604.

\bibitem[KeTa]{tao} M.\ Keel, T.\ Tao, \emph{Endpoint Strichartz estimates}, Amer.\ Math.\ J.\ 120 (1998), pp.\ 955--980.



\bibitem[Mur]{mur} M.\ Murata, \emph{Asymptotic expansions in time for solutions of Schr\"{o}dinger-type equations}, J.\ Funct.\ Anal.\ 49 (1) (1982), pp.\ 10--56. 

\bibitem[Rau]{rau} J.\ Rauch, \emph{Local decay of scattering solutions to Schr\"{o}dingerÕs equation}, Comm.\ Math.\ Phys.\ 61 (1978), no.\ 2, pp.\ 149--168.

\bibitem[Sch]{schlag} W.\ Schlag, \emph{Dispersive estimates for Schr\"{o}dinger operators in dimension two}, Commun.\ Math.\ Phys. (2005), 257 (1), pp.\ 87--117.

\bibitem[Sim]{simon} B.\ Simon, \emph{Schr\"{o}dinger semigroups}, Bull.\ Amer.\ 
Math.\ Soc.\ 7 (1982), pp.\ 447--526.



\bibitem[Ste]{stein} E.\ Stein, \emph{Harmonic Analysis}, Princeton University Press, Princeton, 1994.












\bibitem[Yaj1]{yajima_disp} \emph{Dispersive estimates for Schr\"{o}dinger equations with threshold resonance and eigenvalue}, Commun.\ Math.\ Phys.\ 259, pp.\ 475--509 (2005).

\bibitem[Yaj2]{yajima5} K.\ Yajima, \emph{The $L^p$ boundedness of wave operators for Schr\"{o}dinger operators with threshold singularities.\ I.\ The odd dimensional case}, J.\ Math.\ Sci.\ Univ.\ Tokyo 13 (2006), no.\ 1, pp.\ 43--93.




\end{thebibliography}
\end{document}